\documentclass[11pt,letterpaper]{amsart}
\usepackage[utf8]{inputenc}
\usepackage[english]{babel}
\usepackage[T1]{fontenc}
\usepackage{fancyhdr}
\usepackage[left=2.8cm,right=2.8cm,top=3cm,bottom=2cm]{geometry}
\usepackage{bbm}
\usepackage{lmodern}
\usepackage{amsthm}
\usepackage{amsmath}
\usepackage{faktor}
\usepackage{hyperref}
\usepackage{relsize}
\usepackage{amssymb}
\usepackage{tikz} 
\usepackage{enumitem}
\usepackage{varioref}
\usepackage{minibox}
\usepackage{float}
\usepackage{tabularx}
\usetikzlibrary{knots}
\usetikzlibrary{decorations.markings,hobby,knots,celtic}
\newtheorem{theorem}{Theorem}[section]
\newtheorem{lemma}[theorem]{Lemma}
\newtheorem{remark}[theorem]{Remark}
\newtheorem{definition}[theorem]{Definition}
\newtheorem{proposition}[theorem]{Proposition}
\newtheorem{corollary}[theorem]{Corollary}

\title{The {G}ilmer-{M}asbaum map is not injective on the skein module}
\date{}
\author{EDWIN KITAEFF}

\begin{document}

\begin{abstract}
    we compute the image of the evaluation map for a family of mapping tori of the $2$-torus and find an infinite family for which the restriction of the evaluation map to certain homology classes is not injective. In Section \ref{Sec:Basis} (resp. Section \ref{SubSec:S_Matrix}), we give a basis for the KBSM in the case of the mapping torus of a power of the Dehn twist along the meridian (resp. the mapping torus associated with the matrix $\begin{pmatrix} 0 & -1 \\ 1 & 0 \end{pmatrix}$) of the $2$-torus. This shows that new representations of the KBSM will have to be used in the future if one wants to keep pursuing this method.
\end{abstract}
\maketitle
\section{Introduction}
Let $M$ be a closed oriented 3-manifold. The Kauffman bracket skein module $S_A (M,\mathbb{Q} (A))$ over $\mathbb{Q} (A)$ (aka. KBSM) was introduced independently by Przytycki (\cite{Prz91}) and Turaev (\cite{Tur}). The skein module $S_A (M,\mathbb{Q} (A))$ is defined as the $\mathbb{Q} (A)$-vector space spanned by the framed links in $M$ over the ground field $\mathbb{Q} (A)$ modulo isotopies and the Kauffman skein relations:
\begin{center}
\includegraphics[scale=0.2]{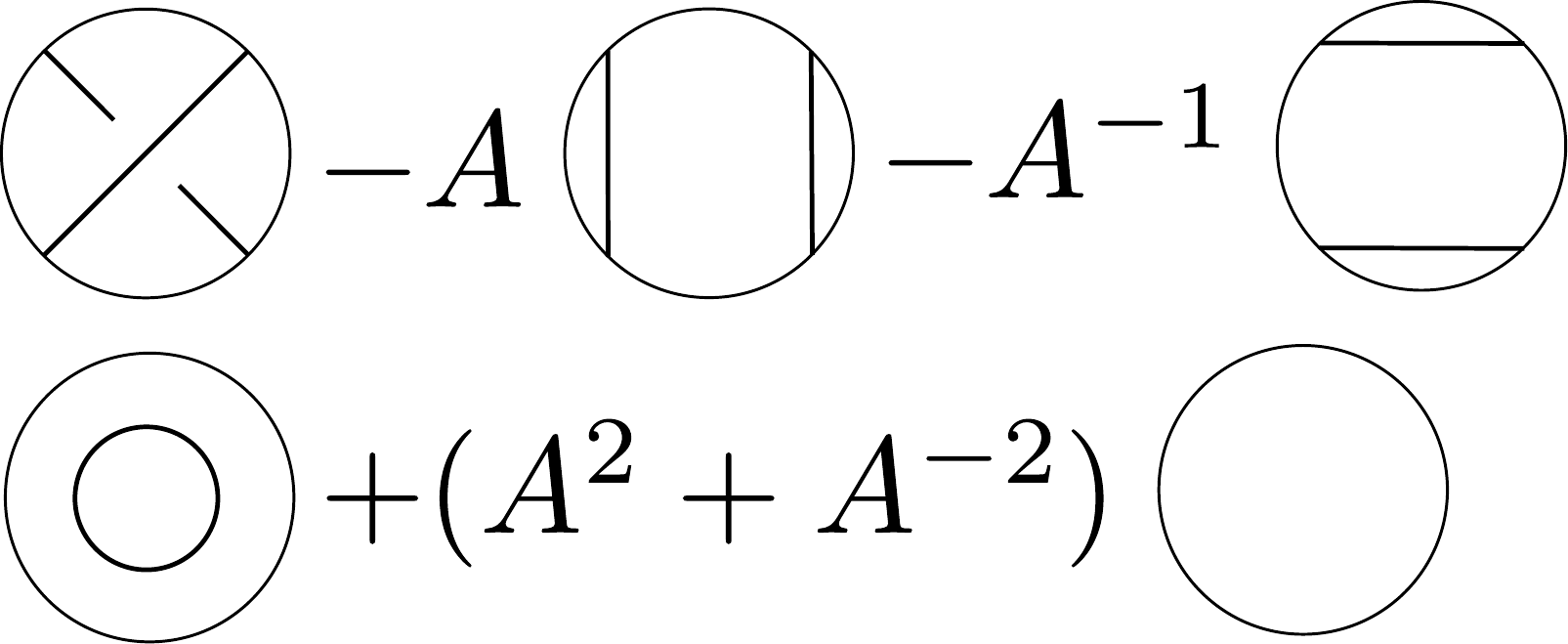}
\end{center}

It was shown in \cite{GJS} that the bracket skein module over $\mathbb{Q} (A)$ is finite dimensional for every closed 3-manifold. Recently, another proof of this fact was provided in \cite{BelDet}, in a more constructive and elementary way. However, even if \cite{BelDet} provide a set of generators, it does not give the dimension of $K(M)$, nor a basis.

Before these two results, other approaches were used, one of them is in \cite{GilMas} in which Gilmer and Masbaum introduced a technique that gave a lower bound for the dimension of the skein module of $\Sigma_g \times \mathbb{S}^1$. Later, in \cite{DetWol}, it was found that this lower bound was optimal, providing a basis for the skein module. 

\subsection{The Gilmer-Masbaum evaluation map}

The technique of \cite{GilMas} is relying on the Witten-Reshetikhin-Turaev invariants $RT_{\xi} (M ,L)$  of a framed link $L$ in $M$ associated to $\xi$ a primitive root of even degree. In particular, $RT_\xi (M,L)$ only depends on the skein class of $L$. Below is a more detailed description of this method.
The Gilmer--Masbaum evaluation map is defined as follows. Denote by
\[
\mathbb{U}_0 := \left\{ e^{\frac{is\pi}{r}} \mid r > 1, \gcd(s,2r) = 1 \right\}
\]
the set of primitive roots of unity of even order, and by $\mathbb{C}^{\mathbb{U}_0}_{\text{a.e.}}$ the set of complex-valued functions that are defined almost everywhere on $\mathbb{U}_0$ (\textit{i.e.}, everywhere except at a finite number of points) for which we identify two maps if they have the same image almost everywhere. For $\lambda_1, \ldots, \lambda_n \in \mathbb{Q}(A)$ and $f_1, \ldots, f_n \in \mathbb{C}^{\mathbb{U}_0}_{\text{a.e.}}$, one can define the element
\[
\sum_{j=1}^n \lambda_j f_j = \left( \xi \mapsto \sum_{j=1}^n \lambda_j(\xi) f_j(\xi) \right) \in \mathbb{C}^{\mathbb{U}_0}_{\text{a.e.}}.
\]
This is well-defined since $\lambda_j$ is defined everywhere except at its poles. This gives $\mathbb{C}^{\mathbb{U}_0}_{\text{a.e.}}$ the structure of a $\mathbb{Q}(A)$-vector space.

Here is the evaluation map:

\begin{definition}[\cite{GilMas}]
\label{def:ev}
\[
\begin{array}[t]{lrcl}
ev_M \colon & S(M,\mathbb{Q} (A))  & \longrightarrow & \mathbb{C}^{\mathbb{U}_0}_{\text{a.e.}} \\
& \displaystyle \sum_{j=1}^n \lambda_j L_j & \longmapsto & \left( \begin{array}{lrcl}
\mathbb{U}_0 & \longrightarrow & \mathbb{C} \\
\xi & \longmapsto & \displaystyle \sum_{j=1}^n \lambda_j(\xi) RT_\xi(M,L_j)
\end{array} \right)
\end{array}
\]
\end{definition}

This original map of \cite{GilMas} is obtained by replacing $\mathbb{U}_0$ by $\left\{ e^{\frac{is\pi}{r}} \mid r > 1\ \textrm{odd}, \gcd(s,2r) = 1 \right\}$. Their result is the following:

\begin{theorem}[{\cite[Theorem 1]{GilMas}}]
\label{thm:GilMas}
Let $\Sigma_g$ be the closed orientable surface of genus $g$. Then:
\[
\dim_{\mathbb{Q}(A)} S_A(\Sigma_g \times \mathbb{S}^1, \mathbb{Q}(A)) \geq 2^{2g+1} + 2g - 1.
\]
\end{theorem}

This result was later completed by the following
\begin{theorem}\cite{DetWol}\label{Thm:DetWolff}
We have the following upper bound:
\[
\dim_{\mathbb{Q}(A)} S_A(\Sigma_g \times \mathbb{S}^1, \mathbb{Q}(A)) \leq 2^{2g+1} + 2g - 1.
\]
\end{theorem}

Since the skein relations are homogeneous, the skein module admits a decomposition into graded subspaces :
$$ S_A (M,\mathbb{Q} (A)) = \underset{x\in H_1 (M,\faktor{\mathbb{Z}}{2\mathbb{Z}})}{\bigoplus} S_x (M)$$
where $S_x (M)$ is the skein module of homology $x$.

The study of linearly independent families in $S_A (M,\mathbb{Q} (A))$ transposes to the study of independancy of families in $\mathbb{C}^{\mathbb{U}_0 }_{a.e.}$ that are images of families in $S_x (M)$ through the computation of the WRT-invariants. In \cite{GilMas}, this technique yielded linearly independent families for $S_A (\Sigma ,\mathbb{Q} (A))$. 

However, a key ingredient for this lower bound to be optimal is that $ev_M$ has to be injective on the graded subspaces. A natural question then, is whether or not the application was injective on each $S_x (M)$.

We will show that this map is not always injective and therefore the Gilmer-Masbaum evaluation map cannot always give an optimal dimension.

\subsection*{The mapping tori}\label{SubSec:MappingTori}

Note by $\Gamma (\mathbb{T}^2 )$ the mapping class group of $\mathbb{T}^2$. The manifolds that will interest us come from the following family:

\begin{definition}\label{def:MappingTori}
For $B\in \Gamma (\mathbb{T}^2) \cong SL_2 (\mathbb{Z} )$, the mapping torus of the $2$-torus with monodromy $B$ is
$$
M_B = \faktor{\mathbb{T}^2 \times [0,1]}{(x,0)\sim (B(x),1)}.
$$
\end{definition}

Let $B_k\in \Gamma (\mathbb{T}^2 )$ be the mapping class of $\begin{pmatrix}
1 & k \\
0 & 1 \\
\end{pmatrix} $, which represents a power of a Dehn twist, and let $M_k := M_{B_k}$. Since $M_{-k}$ is $M_k$ with the orientation reversed, and since $M_0=\mathbb{T}^2\times \mathbb{S}^1$ was treated in \cite{Gil}, we only consider $k\geq 1$. We obtain particularly interesting results when $k$ is square-free.

To understand the decomposition $S_A (M_k , \mathbb{Q} (A)) = \bigoplus_{w} S_w (M_k)$, let $(\alpha , \beta )$ be a basis of $H_1 (\mathbb{T}^2 ,\faktor{\mathbb{Z}}{2\mathbb{Z}})$ such that $\alpha $ is the class of the curve along which the Dehn twist is performed, let $(x,y)$ be the image of $(\alpha , \beta )$ under the map induced by the inclusion $\mathbb{T}^2  \rightarrow M_k$, and let $z$ be the class of the curve $\{ * \} \times \mathbb{S}^1$.

\begin{lemma}\label{lemma:Homology}
$H_1 (M_k ,\faktor{\mathbb{Z}}{2\mathbb{Z}}) = \{ \emptyset , x , y , xy , z , xyz , xz , yz \}$ if $k$ is even, and $H_1 (M_k ,\faktor{\mathbb{Z}}{2\mathbb{Z}}) = \{ \emptyset, y , z , yz \}$ if $k$ is odd.
\end{lemma}

\begin{proof}
Isotoping $y$ along the $\mathbb{S}^1$ factor gives $y=kx+y$; when $k$ is odd, this means $x=\emptyset$. It is not hard to check (see \cite[Lemma 2.3]{Kin}) that $H_1 (M_k ,\faktor{\mathbb{Z}}{2\mathbb{Z}})$ is $\left( \faktor{\mathbb{Z}}{2\mathbb{Z}} \right)^3$ if $k$ is even and $\left( \faktor{\mathbb{Z}}{2\mathbb{Z}} \right)^2$ if $k$ is odd, and it is generated by $\{ x,y,z \}$.
\end{proof}

\subsection*{The horizontal part}\label{sec:HorizontalPart}

Since the fibration over $\mathbb{S}^1$ induces a map $M_k \rightarrow \mathbb{S}^1$, we can define the \textit{horizontal part} of $M_k$ to be
$$ h_k := Ker ( H_1 (M_k ,\faktor{\mathbb{Z}}{2\mathbb{Z}} ) \rightarrow H_1 (\mathbb{S}^1 ,\faktor{\mathbb{Z}}{2\mathbb{Z}}) ),$$
and the \textit{horizontal part} of $S_A (M_k ,\mathbb{Q} (A) )$ to be $\mathcal{H}_k := \bigoplus_{\alpha\in h_k } S_\alpha (M_k)$. Kinnear computed its dimension in \cite{Kin}:

\begin{theorem}\cite[Theorem 1.2]{Kin}\label{Thm:Kinnear}
The dimension of $\mathcal{H}_k$ is $\dfrac{k-1}{2} + 4$ if $k$ is odd and $\dfrac{k}{2} +5$ if $k$ is even.
\end{theorem}

According to \cite[Corollary 1.7]{GJV}, the KBSM of a $3$-manifold containing an embedded $2$-torus is spanned by skein elements that can be represented by links that intersect the torus at most once. In our case, $\mathcal{H}_k$ is therefore spanned by skeins that can be represented by links in $\mathbb{T}^2 \times [\frac{1}{4} ,\frac{3}{4}] \subset M_k$, and we will use the Frohman--Gelca basis (see \cite{FroGel}) to describe its elements.

\subsection*{The results}

In Section \ref{Sec:Basis}, we compute a basis for $\mathcal{H}_k$:

\begin{theorem}\label{Thm:basis}
\begin{itemize}
\item If $k$ is odd, $\{ (p,0)_T \ \vert \ 0\leq p \leq \left\lfloor \dfrac{k}{2} \right\rfloor \} \cup \{ (0,1)_T ,(0,2)_T ,(1,2)_T \}$ is a basis of $\mathcal{H}_k$.
\item If $k$ is even, $\{ (p,0)_T \ \vert \ 0\leq p \leq \dfrac{k}{2} \} \cup \{ (0,1)_T ,(0,2)_T ,(1,2)_T , (1,1)_T \} $ is a basis of $\mathcal{H}_k$.
\end{itemize}
\end{theorem}

\begin{remark}\label{Rk:BasisHorizontal}
With the notation of Lemma \ref{lemma:Homology}, the graded subspaces that are not in $\mathcal{H}_k$ are $S_{z}, S_{x+z}, S_{y+z}, S_{x+y+z}$ when $k$ is even, and $S_{z}, S_{y+z}$ when $k$ is odd. By \cite[Proposition 4.2]{Kin}, each of them has dimension $1$, generated by the natural element. Therefore, Theorem \ref{Thm:basis} provides a full basis of $S_A (M_k ,\mathbb{Q} (A) )$.
\end{remark}

The main theorem of the chapter is the following. We say that $ev$ is \emph{not injective} if there exists $w\in h_k$ such that $ev_{M_k}$ is not injective on $S_w (M_k)$.

\begin{theorem}\label{Thm:Main}
Let $k\geq 1$.
\begin{enumerate}
\item $ev_{M_k}$ is not injective if $k$ is odd with at least $2$ distinct prime divisors, if $k\equiv 0\pmod{4}$ and $k\neq 4$, or if $k\equiv 2 \pmod{4}$ and $k$ has at least $3$ distinct prime divisors.
\item $ev_{M_k}$ is injective on $S_w (M_k)$ for all $w\in h_k$ if $k\in\{1,2,4\}$, or if $k=p$ or $k=2p$ for $p$ an odd prime number.
\end{enumerate}
\end{theorem}

\begin{remark}\label{Remark:unknown}
The only cases for which we do not know whether $ev_{M_k}$ is injective on the graded subspaces are $k=p^{\alpha}$ and $k=2p^{\alpha}$ with $p$ an odd prime and $\alpha>1$. In these cases the upper bound of Theorem \ref{Thm:DimImages} below equals $\dim\mathcal{H}_k$, so the comparison of dimensions is inconclusive (and the bound is not known to be sharp, since $k$ is not square-free). The case $k=4$ is treated by hand in Section \ref{Sec:Concl}, and the case $k=0$ is already proven in \cite{Gil}.
\end{remark}

To prove the non-injectivity (or the injectivity) of the evaluation map, we compute the following. For $k=\prod_{j=1}^{n} p_{j}^{\alpha_j}$ with primes $p_1 < \ldots < p_n$ and integers $\alpha_j > 0$ (and $n=0$ if $k=1$), set
\begin{equation}\label{eq:Theta}
\mathcal{D}(k) := \begin{cases}
\displaystyle \prod\limits_{j=1}^{n} \dfrac{p_j^{\alpha_j} +1 }{2} & \textrm{if } k\textrm{ is odd},\\[3ex]
\displaystyle 2 \prod\limits_{j=2}^{n} \dfrac{p_j^{\alpha_j} +1 }{2} & \textrm{if } k\equiv 2 \pmod{4},\\[3ex]
\displaystyle \Theta (\alpha_1 ) \prod\limits_{j=2}^{n} \dfrac{p_j^{\alpha_j} +1 }{2} & \textrm{if } k\equiv 0 \pmod{4},
\end{cases}
\end{equation}
where $\Theta (\alpha ) := 2+2^{\alpha -2}$ if $\alpha\in\{2,3\}$ and $\Theta (\alpha ) := 1+3\cdot 2^{\alpha -3}$ if $\alpha\geq 4$.

\begin{theorem}\label{Thm:DimImages}
For every $k\geq 1$,
\begin{align*}
\sum\limits_{w \in h_k} \dim (ev(S_w (M_k ) ) ) &\leq \mathcal{D}(k) + 3 && \textrm{if } k \textrm{ is odd},\\
\sum\limits_{w \in h_k} \dim (ev(S_w (M_k ) ) ) &\leq \mathcal{D}(k) + 4 && \textrm{if } k \textrm{ is even}.
\end{align*}
Moreover, these inequalities are equalities if $k$ is square-free.
\end{theorem}

Comparing these bounds (see Lemma \ref{Lemma:Inequalities}) with Theorem \ref{Thm:Kinnear} will give Theorem \ref{Thm:Main}. As explained earlier, Theorem \ref{Thm:DimImages} relies on the computation of the WRT-invariants, which, as we will see in Section \ref{Sec:ev}, are directly related to the generalised quadratic Gauss sums. For $\zeta$ a root of unity of order $c$ and integers $a,b\geq 0$, the generalised quadratic Gauss sum $G(a,b,\zeta )$ is defined by:
\begin{equation}\label{defsum}
G(a,b,\zeta ) = \sum\limits_{j=0}^{c-1}  \zeta^{aj^2 + bj}.
\end{equation}
We consider $G(a,b,\cdot )$ as a function on $\mathbb{U} := \{ -\zeta \ \vert \ \zeta\in \mathbb{U}_0 \}$ (see Lemma \ref{Lemma:U=U0}).


We claim that the study of the image of the evaluation map reduces to the study of the vector space generated by the (normalised) Gauss sums.

\begin{definition}\label{Def:Gk}
For $u\in\mathbb{U}$ of order $c$, let $\delta (u) := 1$ if $c$ is odd and $\delta (u):=2$ if $c$ is even, and set $\widehat{G}:=\frac{G}{\delta}$. We define
$$\mathcal{G}_{k} = Span_{\mathbb{Q} (A) } \left( \left\lbrace  \widehat{G}(k,2l,\cdot ) \right\rbrace_{l\in \{0,\ldots ,k-1\} }  \right).$$
\end{definition}

Since $\delta$ never vanishes, $\mathcal{G}_k$ has the same dimension as the span of the functions $G(k,2l,\cdot)$. That generalised quadratic Gauss sums appear here is not a surprise, as is often the case when computing quantum invariants. Thus, the following study of their span may be useful outside the context of this chapter:

\begin{theorem}\label{Thm:dimG}
For every $k\geq 1$, $\dim_{\mathbb{Q}(A)}\mathcal{G}_k\leq\mathcal{D}(k)$, with equality if $k$ is square-free.
\end{theorem}


\begin{remark}
We will see that $\mathcal{D}(k)$ is close to the sum, over the divisors $d$ of $k$, of the number of invertible squares modulo $\frac{k}{d}$.
\end{remark}

In Section \ref{Sec:Basis}, we find a basis of the horizontal part of $M_k$ using the Frohman--Gelca product-to-sum formulas. In Section \ref{Sec:ev}, we compute the WRT-invariants using TQFT properties, and in Section \ref{Sec:GaussSums} we study the dimension of $\mathcal{G}_{k}$. We conclude in Section \ref{Sec:Concl}. Moreover, we discuss how to compute the image of $ev_{M_B}$ for a general monodromy $B$ (without fully carrying out the computations) in Section \ref{sec:GeneralFormulas}, and we apply this to the monodromy $\begin{pmatrix} 0 & -1 \\ 1 & 0 \end{pmatrix}$ in Section \ref{SubSec:S_Matrix}.

Only Section \ref{Sec:SquareFree} and the equality cases of Section \ref{Sec:Concl} require $k$ to be square-free; in particular, the basis of Section \ref{Sec:Basis} works for any $k$, and we still get an upper bound for the image of the evaluation map by Theorem \ref{Thm:DimImages}. Therefore, we also show that $ev_{M_k}$ is not injective for some non square-free $k$, as soon as the upper bound is too low.

\section{A basis for the horizontal part \texorpdfstring{$\mathcal{H}_k$}{H\_k}}\label{Sec:Basis}

According to Theorem \ref{Thm:Kinnear}, the set described in Theorem \ref{Thm:basis} has $\dim_{\mathbb{Q} (A)} \mathcal{H}_k$ elements. It is therefore sufficient to show that it spans $\mathcal{H}_k$. To do so, we first give the following relations:

\begin{lemma}\label{Lemma:Formulas}
Let $p,q\in\mathbb{Z}$. In $\mathcal{H}_k$,
\begin{equation}\tag{a}\label{Dependance1}
(p+1,q)_T = (p-1,q)_T \qquad\textrm{if } q\neq 0,
\end{equation}
\begin{equation}\tag{b}\label{Dependance2}
A^{p} (p,q+1)_T + A^{-p} (p,q-1)_T = A^{kq-p} (p+k,q+1)_T + A^{p-kq} (p-k,q-1)_T.
\end{equation}
\end{lemma}

\begin{proof}
Isotoping a skein element $(p,q)_T \in \mathcal{H}_k$ along the $\mathbb{S}^1$ factor gives $(p,q)_T = B_k ((p,q)_T) = (p+qk, q)_T$. Applying this to $(1,0)_T$, which is invariant under $B_k$, we get $(p,q)_T \star (1,0)_T = (1,0)_T \star (p,q)_T$. By the product-to-sum formula,
$$A^{-q} (p+1,q)_T +A^q (p-1,q)_T  = A^q (p+1,q)_T +A^{-q} (p-1,q)_T,$$
that is, $(A^{-q} - A^q ) \left((p+1,q)_T - (p-1 ,q)_T\right)=0$, which gives (\ref{Dependance1}) when $q\neq 0$. On the other hand, $B_k ((0,1)_T ) = (k,1)_T$ gives $(p,q)_T \star (0,1)_T = (k,1)_T \star (p,q)_T$, and the product-to-sum formula together with $(p,q)_T = (-p,-q)_T$ gives (\ref{Dependance2}).
\end{proof}

\begin{proof}[Proof of Theorem \ref{Thm:basis}]
Suppose that $k$ is odd. Formula (\ref{Dependance1}) gives
\begin{equation}\label{ligne}\tag{e1}
\forall q>0,\ \forall p\in\mathbb{Z},\ (p,q)_T \in Span_{\mathbb{Q} (A)} \{ (0,q)_T , (1,q)_T \}.
\end{equation}
Because $(p,q)_T=(-p,-q)_T$, the set $\{ (p,q)_T , \ q\geq 0 \}$ spans $\mathcal{H}_k$, and (\ref{ligne}) leaves us with the elements $(p,0)_T$ and $(p,q)_T$ with $q>0$, $p\in \{ 0,1\}$.

Consider $p\in\{ 0,1\}$ and $q>1$. Since $k$ is odd, injecting (\ref{ligne}) into (\ref{Dependance2}) gives
\begin{equation}\label{kodd}\tag{e2}
    A^{p} (p,q+1)_T  = A^{kq-p} (1-p ,q+1)_T + A^{p-kq} (1-p ,q-1)_T - A^{-p} (p,q-1)_T.
\end{equation}
Applying this equality again to $(1-p ,q+1)_T$ and substituting, we get
$$
(1- A^{2 kq-2} ) (p ,q+1)_T = (A^{-kq} - A^{kq-2} ) (1-p ,q-1)_T.
$$
Since $q>1$, we have $1- A^{2 kq-2}  \neq 0$, so $(p,q+1)_T$ lies in the span of $(0,q-1)_T,(1,q-1)_T$. By induction, and since $(1,1)_T=B_k ((1,1)_T )= (k+1 ,1)_T=(0,1)_T$ by (\ref{Dependance1}), we obtain that $\mathcal{H}_k$ is spanned by $\{ (p,0)_T \ \vert\ p\in \mathbb{Z} \} \cup \{ (0,1)_T ,(0,2)_T ,(1,2)_T \}$.

Applying (\ref{Dependance2}) with $q=1$ expresses $(p-k,0)_T$ in terms of $(p,0)_T$, $(p,2)_T$ and $(p+k,2)_T$. After sufficiently many uses of this relation, together with $(-p,0)_T=(p,0)_T$ and (\ref{ligne}), every $(p,0)_T$ lies in the span of $\{ (l,0)_T \ \vert\ 0\leq l\leq \lfloor k/2\rfloor \} \cup \{ (0,2)_T,(1,2)_T\}$. This proves the case $k$ odd.

If $k$ were even, the only two differences would be that $(1,1)_T$ would remain necessary in the generating set, and that we would have $p$ instead of $1-p$ in (\ref{kodd}), which would end that step directly.
\end{proof}

\section{Computation of \texorpdfstring{$ev_{M_k}$}{ev}}\label{Sec:ev}

Let $\zeta=e^{\frac{is\pi}{r}}\in\mathbb{U}_0$ be a primitive root of unity of order $2r$ with $r>1$, and set $u:=-\zeta$. Here $r$ is the integer such that $\zeta$ has order $2r$, whether $r$ is odd or even; when $r$ is even, we write $m=\frac{r}{2}$, so that $\zeta$ is a primitive $4m$-th root of unity.

In this section, we use the invariants $RT_\zeta$, which give rise to TQFTs via the universal construction. Since we will only be interested in the linear independence of a family given by some $\{ RT_\zeta (M_k , L) \}$, we ignore the anomaly $\kappa$, which is a power of $\zeta$ which does not depend on $L$, and we fix the Lagrangian in $H_1 (\mathbb{T}^2 , \mathbb{Q} )$ to be the subspace generated by the class of the meridian, so that we no longer refer to it.

Let us make the basis of $RT_\zeta (\mathbb{T}^2 )$ from \cite[Corollary 4.10]{BHMVInvariant} explicit. If $r$ is odd, $\{e_j\}_{1\leq j\leq \frac{r-1}{2}}$ is a basis of $RT_\zeta(\mathbb{T}^2)$, and $e_{r-j}=e_j$, $e_{j+r}=-e_j$. If $r$ is even, $\{e_j\}_{1\leq j\leq m-1}$ is a basis, $e_m=0$, $e_{r-j}=-e_j$ and $e_{j+r}=e_j$. In both cases $e_{-j}=-e_j$. By the trace formula,
$$ev ( (p,q)_T ) (\zeta )= RT_\zeta (M_k , (p,q)_T ) = Tr_{RT_\zeta (\mathbb{T}^2)} (\rho (B_k ) \circ Z((p,q)_T ))$$
with $\rho (B_k )$ the element of $End (RT_\zeta (\mathbb{T}^2))$ corresponding to the mapping cylinder of $B_k$ and $Z((p,q)_T )$ the curve operator of the curve $(p,q)_T$. We therefore compute these two operators on the basis $\{e_j\}$.

\begin{lemma}\label{Lemma:rho}
The action of $\rho (B_k )$ on the $e_j$ is $\rho (B_k) (e_j ) = (-\zeta )^{k(j^2 -1)} e_j$. Moreover, this formula holds for all $j\in\mathbb{Z}$.
\end{lemma}

\begin{proof}
By noticing that $\rho (B_1)$ is the effect of a simple Dehn twist, one recognizes the formula of \cite[p.690-691]{BHMVInvariant}: $\rho (B_1) (e_j) = (-\zeta )^{j^2 -1} e_j$. Thus $\rho (B_k ) (e_j) = \rho (B_1 )^k (e_j) =(-\zeta )^{k(j^2 -1)} e_j$. The formula holds for all $j\in\mathbb{Z}$ because $u^{k(j^2-1)}$ is unchanged when $j$ is replaced by $-j$ or $j+r$ (use $u^{2r}=u^{kr^2}=1$), which are the operations used to bring $e_j$ back to the basis.
\end{proof}

\begin{lemma}\label{Lemma:Action(p,q)}
Let $p,q$ be integers. The action of $Z((p,q)_T )$ on the generating set $\{ e_j \}$ is:
$$
Z((p,q)_T ) e_j = (-1)^{p} (\zeta^{2pj +pq} e_{j+q} + \zeta^{-2pj+pq} e_{j-q})
$$
\end{lemma}

\begin{proof}
If one considers the morphism $Z : S_A (\mathbb{T}^2 , \mathbb{Q} (A) )\rightarrow Hom(RT_{\zeta} (\mathbb{T}^2 ) ,RT_{\zeta} (\mathbb{T}^2 ))$, since $(0,1)_T$ and $(1,0)_T$ span $S_A (\mathbb{T}^2 , \mathbb{Q} (A))$ as an algebra, it suffices to compute the actions of $Z ((0,1)_T )$, $Z((1,0)_T )$ and their compositions on the generating set $\{ e_j \}$ to show that it coincides with the above formula.

First, it is proven in \cite[p.690-691]{BHMVInvariant} that $$Z((1,0)_T ) e_j  = -(\zeta^{2j} + \zeta^{-2j} ) e_j$$
Moreover, $Z((0,1)_T ) e_j = z e_j = e_{j+1} + e_{j-1}$. And now, if, for some $p,q,s,t\in\mathbb{Z}$ and for all $j$, $$Z((s,t)_T ) e_j = (-1)^{s} (\zeta^{2sj +st} e_{j+t} + \zeta^{-2sj+st} e_{j-t})$$ and $$Z((p,q)_T ) e_j = (-1)^{p} (\zeta^{2pj +pq} e_{j+q} + \zeta^{-2pj+pq} e_{j-q})$$ then a direct computation gives that:
\begin{align*}
&Z((p,q)_T ) \circ Z((s,t)_T ) e_j \\
&= (-1)^{p+s} ( \zeta^{pt -sq} \zeta^{2(p+s)j +(p+s)(q+t)} e_{j+t+q} + \zeta^{pt -sq} \zeta^{-2(p+s)j + (p+s)(q+t)} e_{j-t-q})\\
& + (-1)^{p-s} ( \zeta^{sq- pt} \zeta^{-2(p-s)j + (p-s)(q-t)} e_{j+t-q} + \zeta^{sq- pt} \zeta^{2(p-s)j +(p-s)(q-t)} e_{j-t+q}  )
\end{align*} 
which is what is expected of $$Z((p,q)_T \star (s,t)_T ) e_j = \zeta^{pt-qs} Z((p+s,q+t)_T ) e_j + \zeta^{qs-pt} Z((p-s,q-t)_T ) e_j.$$
\end{proof}

This gives the elementary blocks for the computation of the evaluation map. Recall that $u=-\zeta$, $2r=ord(\zeta )$, and $m=\frac{r}{2}$ when $r$ is odd. For $\epsilon\in\{0,1\}$, define the function
\begin{equation}\label{eq:nu}
    \nu_\epsilon (\zeta) := \begin{cases} 1 & \textrm{if } r \textrm{ is odd,}\\ 1+(-1)^{\epsilon}u^{km^2} & \textrm{if } r \textrm{ is even.}\end{cases}
\end{equation}
Note that $\nu_0=\nu_1=1$ when $r$ is odd, while $u^{km^2}$ is a power of $i$, so $\nu_0$ and $\nu_1$ differ when $r$ is even.

\begin{lemma}\label{Lemma:EvCasParticuliers}
Let $p,q$ be integers with $q>0$ (recall that $(p,q)_T=(-p,-q)_T$), and let $\zeta\in\mathbb{U}_0$ be such that $r>q$ (recall that, in the definition of $ev$, we may discard finitely many roots of unity $\zeta$).

If $q$ is even, then $ev((p,q)_T)(\zeta)=(-1)^{p+1}u^{k((\frac{q}{2})^2 -1)}\nu_\epsilon(\zeta)$, with $\epsilon\equiv p+\frac{kq}{2} \pmod 2$. In particular, $ev((p,2)_T)=(-1)^{p+1}\nu_{p+k \bmod 2}$.

If $q$ is odd, then $ev ( (p,q)_T ) (\zeta ) = 0$ if $r$ is even and $ev ( (p,q)_T ) (\zeta ) =u^{k((\frac{r+q}{2} )^2 -1) }$ if $r$ is odd. In particular, $ev((p,q)_T)\neq 0$.
\end{lemma}

\begin{proof}
We only carry out the computations when $q$ is even; when $q$ is odd, the proof is similar. By the trace formula and Lemmas \ref{Lemma:rho} and \ref{Lemma:Action(p,q)}, which hold for all indices,
$$\rho (B_k ) \circ Z((p,q)_T ) e_j = (-1)^{p} (\zeta^{2pj+pq}u^{k((j+q)^2-1 )} e_{j+q} + \zeta^{-2pj +pq} u^{k((j-q)^2 -1 )} e_{j-q} ).$$
The trace is the sum, over the basis, of the coefficient of $e_j$ in this expression, after bringing $e_{j\pm q}$ back to the basis. So we look for the $j$ such that $e_{j\pm q}=\pm e_j$.

Since $q$ is even and $0<q<r$, the only solution of $j-q\equiv\pm j \pmod r$ is $j=\frac{q}{2}$, for which $e_{j-q}=-e_j$; it contributes $-(-1)^p u^{k((\frac{q}{2})^2-1)}$. If $r$ is odd, there is no other contribution, which gives the first case. If $r$ is even, $j+q\equiv -j \pmod r$ has the solution $j=m-\frac{q}{2}$, for which $e_{j+q}=e_{r-j}=-e_j$, and it contributes $-(-1)^p\zeta^{pr}u^{k((m+\frac{q}{2})^2-1)}=-u^{k((m+\frac{q}{2})^2-1)}$. Since $(m+\frac{q}{2})^2=m^2+mq+(\frac{q}{2})^2$, the trace is
$$u^{k((\frac{q}{2})^2 -1)}\left( (-1)^{p+1} - u^{k(m^2+mq)}\right).$$
Finally $u^{m}=\pm i$, so $u^{kmq}=(-1)^{\frac{kq}{2}}$, and the expression above is $(-1)^{p+1}u^{k((\frac{q}{2})^2 -1)}\nu_\epsilon(\zeta)$.
\end{proof}

In order to have a complete computation of the evaluation map on $\mathcal{H}_k$, it remains to study the case $q=0$. First, we describe the set $\mathbb{U}$.

\begin{lemma}\label{Lemma:U=U0}
$\mathbb{U} := \{ -\zeta \ \vert \ \zeta\in \mathbb{U}_0 \} = \{ e^{\frac{2 is\pi }{c}} \ \vert\ c > 2,\ gcd(s,c) = 1,\ c \not\equiv 2 \pmod{4} \}$.
\end{lemma}

\begin{proof}
When $\zeta$ is a primitive $2r$-th root of unity, $-\zeta$ is a primitive root of unity of order $r$ if $r$ is odd, and a primitive $2r$-th root of unity if $r$ is even. Conversely, when $\zeta$ is a primitive $c$-th root of unity with $c>2$ and $c\not\equiv 2 \pmod{4}$, $-\zeta$ is a primitive root of unity of order $2c$ if $c$ is odd, and of order $c$ if $c$ is even.
\end{proof}

Here is the computation of $ev ((l,0)_T)$:

\begin{theorem}\label{Thm:CalculRT}
Let $\zeta=e^{\frac{is\pi}{r}} \in \mathbb{U}_0$ and $u=-\zeta$. For every integer $l\geq 0$,
$$ev ((l,0)_T)(\zeta )= RT_\zeta (M_k, (l,0)_T )= (-1)^l u^{-k} \left( \widehat{G}(k,2l,u ) - \nu_{l \bmod 2} (\zeta) \right).$$
\end{theorem}

\begin{proof}
Again, by the trace formula, $RT_\zeta (M_k , (l,0)_T) = Tr (\rho (B_k ) \circ Z((l,0)_T))$. By Lemmas \ref{Lemma:rho} and \ref{Lemma:Action(p,q)}, this operator is diagonal in the basis $\{e_j\}_{1\leq j\leq N_r}$ (where $N_r =\frac{r-1}{2}$ if $r$ is odd and $N_r =m-1$ if $r$ is even), with eigenvalue $(-1)^l u^{k(j^2 -1)} (u^{2lj} + u^{-2lj} )$ on $e_j$ (recall that $\zeta^2=u^2$). Hence
$$RT_\zeta (M_k , (l,0)_T) = (-1)^l u^{-k}\sum\limits_{\substack{-N_r \leq j\leq N_r \\ j\neq 0}} u^{kj^2+2lj}.$$

First, suppose that $r$ is odd. Then $u$ has order $r$ and $\{-N_r,\ldots,N_r \}$ is a complete set of representatives of $\faktor{\mathbb{Z}}{r\mathbb{Z}}$, so that, compensating for the term $j=0$,
$$RT_\zeta (M_k , (l,0)_T)=  (-1)^l u^{-k} \Big(\sum\limits_{j=0}^{r-1}  u^{kj^2 + 2lj} -1 \Big)= (-1)^l u^{-k}(G(k,2l,u) -1 ).$$
Now suppose that $r$ is even. Then $\{-(m-1),\ldots,m-1\}\cup\{m\}$ is a complete set of representatives of $\faktor{\mathbb{Z}}{r\mathbb{Z}}$. Using the fact that $u$ has order $2r$, we have that
\begin{align*}\sum\limits_{\substack{-(m-1)\leq j\leq m-1\\ j\neq 0}} u^{kj^2+2lj} &= \sum\limits_{j=0}^{r-1} u^{kj^2+2lj} - 1 -u^{km^2+2lm} =\dfrac{G(k,2l,u)}{2}-1-(-1)^l u^{km^2},\end{align*}
because $u^{2lm}=u^{lr}=(-1)^l$ ($u$ has order $2r$, so $u^r=-1$) and $j\mapsto u^{kj^2+2lj}$ is $r$-periodic. This is the announced formula, since $\nu_{l\bmod 2}(\zeta)=1+(-1)^lu^{km^2}$ for $r$ even.
\end{proof}

\section{Dimension of the quadratic Gauss sums over \texorpdfstring{$\mathbb{Q} (A) $}{Q(A)}}\label{Sec:GaussSums}
In this section, we focus on the study of $\mathcal{G}_k$. As in Lemma \ref{Lemma:U=U0}, $\mathbb{U} = \{ e^{\frac{2 is\pi }{c}} \ \vert\ c > 2,\ gcd(s,c) = 1,\ c \not\equiv 2 \pmod{4} \}$ and $\mathbb{U}_{odd}$ is its subset of elements of odd order. 
For coprime integers $a$ and $r$, we write $a^{\star}_{r}$ for \textit{an} inverse of $a$ modulo $r$; only its class modulo $r$ matters, as it always appears in the exponent of a root of unity of order a multiple of $r$. We will make an explicit choice for $a^{\star } (r)$ in Section \ref{Sec:SquareFree}.

\subsection*{Arithmetic preliminaries}

The quadratic Gauss sums being hard to manipulate as they are, we first summarize their expressions:

\begin{lemma}\label{Lemma:CalculGauss}
Let $a$ and $b$ be integers, and let $\zeta = e^{\frac{2is\pi}{qr}}$ with $gcd(s,qr)=1$ and $q=gcd(a,qr)$. Then $gcd(\frac{a}{q},r)=1$, and $\frac{a}{q}$ is odd if $r$ is even. Moreover:
\begin{enumerate}
    \item $G(a,0,\zeta )= \pm q\,\epsilon_r \sqrt{r}$ for some $\epsilon_r \in \mathbb{C}$, with $ \epsilon_r = 0$ if and only if $r\equiv 2 \pmod{4}$; when $r$ is odd, $\epsilon_r$ is a root of unity which only depends on $r \bmod 4$.
    \item $G(a,2b,\zeta) = 0$ if $q$ does not divide $2b$. If $q$ divides $b$, then
$$G(a,2b,\zeta ) = \zeta^{-q (\frac{b}{q})^{2} (\frac{a}{q})^{\star}_{r}} G(a,0,\zeta).$$
    \item If $q$ divides $2b$ but not $b$ and $r\equiv 2 \pmod{4}$, then
$$G(a,2b,\zeta )= (\zeta^{2q})^{- (\frac{2b}{q})^2 (\frac{2a}{q})^{\star}_{\frac{r}{2}} 4^{\star}_{\frac{r}{2}} } G(4a,0,\zeta ) \neq 0.$$
    \item If $q$ divides $2b$ but not $b$ and $r$ is odd, then $G(a,2b,\zeta )=\zeta^{-q (\frac{2b}{q})^2 (\frac{a}{q})^{\star}_{r}4^{\star}_{r} } G(a,0,\zeta )$.
    \item If $q$ divides $2b$ but not $b$ and $r\equiv 0 \pmod{4}$, then $G(a,2b,\zeta )= 0$.
\end{enumerate}
\end{lemma}


\begin{proof}
Item (1) follows from \cite[Theorem 1.5.1, Theorem 1.5.2]{BKE}; the remaining cases are done in an elementary way by completing the square.


\end{proof}

We will see that the dimension of $\mathcal{G}_k$ is related to the number of invertible squares modulo the divisors of $k$. For an integer $n$, we denote by $\Lambda_n$ the set of invertible squares modulo $n$:
\begin{equation}\label{Notation:InvertibleSquares}
\Lambda_n := \left\{ y \in \left( \faktor{\mathbb{Z}}{n\mathbb{Z}}\right)^{\times} \mid \exists\ x \in \left( \faktor{\mathbb{Z}}{n\mathbb{Z}}\right)^{\times} , \ y = x^2 \right\}.
\end{equation}
When $16\mid k$, the collinearity relations between Gauss sums involve a sign (Proposition \ref{Prop:depcarre}), and the relevant modulus is sometimes doubled. For a divisor $d$ of $k$, we therefore set $n_d:=\frac{2k}{d}$ if $d$ is even and $8\mid \frac{k}{d}$, and $n_d:=\frac{k}{d}$ otherwise. Note that $n_d=\frac{k}{d}$ for all $d$ as soon as $16\nmid k$.

\begin{lemma}\label{Lemma:CardDiv}
For every $k\geq 1$, $\sum\limits_{d\mid k} Card(\Lambda_{n_d}) = \mathcal{D}(k)$, where $\mathcal{D}(k)$ is defined in (\ref{eq:Theta}).
\end{lemma}

\begin{proof}
Write $k=2^{\alpha}m$ with $m$ odd, and let $n=\frac{k}{d}=2^{\beta}n'$, with $n'$ odd, run through the divisors of $k$. We have that $n_d=2n$ if $3\leq\beta\leq\alpha -1$ and $n_d = n$ otherwise. If we set the convention $Card(\Lambda_{1} ) =1$, $f : d \longmapsto Card(\Lambda_{n_d})$ is multiplicative; then as a Dirichlet convolution of the two multiplicative maps $f$ and $d \longmapsto 1$ the map $g : k \longmapsto \mathlarger{\sum\limits_{d\vert k}}\ f(d)$ is also multiplicative.

By the Chinese remainder theorem, the sum is the product of $\sum_{\beta =0}^{\alpha}c_{\alpha}(\beta)$ and $\sum_{n'\mid m}Card(\Lambda_{n'})$, where $c_\alpha(\beta)=Card(\Lambda_{2^{\beta+1}})$ if $3\leq\beta\leq\alpha-1$ and $c_\alpha(\beta)=Card(\Lambda_{2^{\beta}})$ otherwise. 

For an odd prime $p$, $Card(\Lambda_{p^{\beta}})=\frac{(p-1)p^{\beta -1}}{2}$ for $\beta\geq 1$, so that $$\sum_{\beta=0}^{\alpha}Card(\Lambda_{p^{\beta}})=1+\frac{p-1}{2}\sum_{\beta=1}^{\alpha}p^{\beta-1}=\frac{p^{\alpha}+1}{2}.$$
For the $2$-part, $Card(\Lambda_{2^{\beta}})=1$ if $\beta\leq 2$ and $2^{\beta-3}$ if $\beta\geq 3$. If $\alpha\leq 3$, we have that $n_d=n$ and the result is straightforward. 
If $\alpha\geq 4$, it is
$$3+\sum_{\beta=3}^{\alpha-1}2^{\beta-2}+2^{\alpha-3}=3+(2^{\alpha-2}-2)+2^{\alpha-3}=1+3\cdot 2^{\alpha-3}=\Theta(\alpha).$$
\end{proof}

Finally, we will need the following elementary inequalities, in order to compare the dimension of $\mathcal{G}_k$ and the dimension of $\mathcal{H}_k$.

\begin{lemma}\label{Lemma:Inequalities}
\begin{enumerate}
\item If $k=\prod_{j=1}^{n} p_{j}^{\alpha_j}$ is odd, then $\mathcal{D}(k)\leq \frac{k+1}{2}$, with equality if and only if $n\leq 1$.
\item If $k=2\prod_{j=1}^{n} p_{j}^{\alpha_j}$ with $p_j$ odd, then $\mathcal{D}(k)\leq \frac{k}{2}+1$, with equality if and only if $n\leq 1$.
\item If $4\mid k$ and $k\neq 4$, then $\mathcal{D}(k)+4<\frac{k}{2}+5$.
\end{enumerate}
\end{lemma}

\begin{proof}
(1) The case $n\leq 1$ is obvious. For odd integers $a,k'\geq 1$ and $P\leq \frac{k'+1}{2}$, we have $\frac{a+1}{2}P\leq \frac{(a+1)(k'+1)}{4}\leq \frac{ak'+1}{2}$, the last inequality being equivalent to $(a-1)(k'-1)\geq 0$ and strict if $a,k'\geq 3$. We conclude by induction on $n$. (2) is (1) applied to $\frac{k}{2}$, multiplied by $2$. (3) Write $k=2^\alpha m$ with $m$ odd. Then case $m=1$ is straightforward. Assume that $m>1$. By (1), $\mathcal{D}(k)\leq \Theta(\alpha)\frac{m+1}{2}$, so it suffices to show $\Theta (\alpha)(m+1)<2^{\alpha}m+2$. For $\alpha\in \{ 2,3 \}$ it is true as $m>1$. For $\alpha\geq 4$, we have that $$\Theta (\alpha ) (m+1) = m (3.2^{\alpha-3} +1) + 1 + 3.2^{\alpha-3} \leq m 2^{\alpha-1} + 1 + m.2^{\alpha-3} < 2^{\alpha}m+2. $$
\end{proof}

\subsection{Upper bound for the dimension of the quadratic Gauss sums}

In this section, we give an upper bound for the dimension of $\mathcal{G}_k$ for a general $k$. The key point is that Gauss sums attached to integers with the same square are collinear, up to a sign which only appears when $16\mid k$.

\begin{proposition}\label{Prop:depcarre}
Let $d$ be a divisor of $k$, and let $l,l'$ be coprime with $k_d:=\dfrac{k}{d}$. Let $m\in\mathbb{Z}$ be such that $l^2 = (l')^2 + m k_d$. Then, for all $\zeta \in \mathbb{U}$ of order $qr$ with $q=gcd(k,qr)$,
$$ G(k,2d l, \zeta) = \epsilon\, \zeta^{-md} G(k,2d l', \zeta), $$
where $\epsilon=(-1)^m$ if $q$ divides $2d$ but not $d$ and $r$ is odd, and $\epsilon=1$ otherwise.
\end{proposition}

\begin{proof}
Let $\zeta=e^{\frac{2is\pi}{qr}}$ with $gcd(s,qr)=1$, and recall $q\mid k$. We apply Lemma \ref{Lemma:CalculGauss} with $a=k$ and $b=dl$ or $dl'$. 
We now enumerate the possible cases for $q$.
\begin{itemize}
\item If $q\nmid 2d$, both sides vanish by item (2) of Lemma \ref{Lemma:CalculGauss}. If $q\mid 2d$, $q\nmid d$ and $4\mid r$, both sides vanish by item (5) of Lemma \ref{Lemma:CalculGauss}.
\item If $q\mid d$, item (2) of Lemma \ref{Lemma:CalculGauss} gives $G(k,2dl,\zeta)=\zeta^{-q(\frac{d}{q})^2l^2(\frac{k}{q})^{\star}_r}G(k,0,\zeta)$, and the same for $l'$. Since $l^2-(l')^2=mk_d$ and $(\frac{d}{q})^2k_d=\frac{d}{q}\frac{k}{q}$, we get
$$G(k,2dl,\zeta)=\zeta^{-q\frac{d}{q}m\frac{k}{q}(\frac{k}{q})^{\star}_r}G(k,2dl',\zeta)=\zeta^{-md}G(k,2dl',\zeta),$$
because $\frac{k}{q}(\frac{k}{q})^{\star}_r\equiv 1 \pmod r$ and $\zeta^q$ has order $r$.
\item If $q\mid 2d$, $q\nmid d$ and $r$ is odd, then $q$ is even and $t:=\frac{2d}{q}$ is odd. Item (4) of Lemma \ref{Lemma:CalculGauss} gives $G(k,2dl,\zeta)=\zeta^{-qt^2l^2(\frac{k}{q})^{\star}_r4^{\star}_r}G(k,0,\zeta)$, and the same for $l'$, so that $G(k,2dl,\zeta)=\zeta^{-x}G(k,2dl',\zeta)$ with $x=q\,t^2mk_d(\frac{k}{q})^{\star}_r4^{\star}_r$. As $k_d=\frac{2k}{qt}$, we have $t^2k_d=2t\frac{k}{q}$, hence $x\equiv q\,tm\,2^{\star}_r \pmod{qr}$, and we may take $2^{\star}_r=\frac{r+1}{2}$. Since $\zeta$ has even order $qr$, $\zeta^{qr/2}=-1$, and since $t$ is odd and $\frac{q}{2}t=d$,
$$\zeta^{x}=\left(\zeta^{\frac{qr}{2}}\right)^{tm}\zeta^{\frac{q}{2}tm}=(-1)^{m}\zeta^{md}.$$
Thus $G(k,2dl,\zeta)=(-1)^m\zeta^{-md}G(k,2dl',\zeta)$.
\item If $q\mid 2d$, $q\nmid d$ and $r=2r'$ with $r'$ odd, set $t=\frac{2d}{q}$ and $v=\zeta^{2q}$, of order $r'$. As $\frac{k}{q}$ is odd, we have that $v_2(k_d)=v_2(q)-v_2(d)=1$, so $l,l'$ are odd, $8\mid l^2-(l')^2=mk_d$ and $m=4m'$. Item (3) of Lemma \ref{Lemma:CalculGauss} gives $G(k,2dl,\zeta)=v^{-t^2l^2(\frac{2k}{q})^{\star}_{r'}4^{\star}_{r'}}G(4k,0,\zeta)$, and the same for $l'$, hence, using $t^2k_d=2t\frac{k}{q}$,
\begin{align*}G(k,2dl,\zeta)&=v^{-tm\,4^{\star}_{r'}}G(k,2dl',\zeta)=v^{-tm'}G(k,2dl',\zeta)\\&=\zeta^{-4dm'}G(k,2dl',\zeta)=\zeta^{-md}G(k,2dl',\zeta),\end{align*}
since $2qt=4d$.
\end{itemize}
\end{proof}

\begin{remark}\label{Remark:SignNeeded}
Since $qr\not\equiv 2 \pmod{4}$, we have that $\epsilon=1$ for every $\zeta$ as soon as $d$ is odd or $v_2(k_d)\leq 2$. In particular, $\epsilon=1$ always if $16\nmid k$.
\end{remark}

\begin{corollary}\label{Cor:UpperBounds}
For every $k\geq 1$, the functions $\widehat{G}(k,2l,\cdot )$, $0\leq l\leq k-1$, fall into at most $\mathcal{D}(k)$ classes of pairwise collinear functions, and, if $k$ is even, the integers $l$ in a given class all have the same parity. In particular, $\dim_{\mathbb{Q}(A)}\mathcal{G}_k\leq\mathcal{D}(k)$, and $\mathcal{G}_k$ is spanned by the $\widehat{G}(k,2l,\cdot )$ with $0\leq l\leq \lfloor \frac{k}{2} \rfloor$.
\end{corollary}
\begin{proof}
Since $\mathcal{G}_k$ has the same dimension as the span of the functions $G(k,2l,\cdot )$, the only points that is not clear considering Proposition \ref{Prop:depcarre} is the statement about the parity.

Two elements on the same class are corresponding to a divisor $d$ of $k$ and $l,l'$ integers so that $gcd(l,k)=gcd(l',k)=d$ and $l^2 \equiv (l')^2 \pmod{k}$. Let us proceed by contrapositive and assume that $l$ and $l'$ have different parity.

Therefore there is integers $q,q'$ such that $(2q)^2 \equiv (2q'+1)^2 \pmod{k}$. Then $4(q^2 - q'^2 - q') \equiv 1 \pmod{k}$.

It follows that $4$ is invertible modulo $k$ and therefore $k$ is odd.

\end{proof}

By Lemma \ref{Lemma:CardDiv} and Corollary \ref{Cor:UpperBounds}, this gives the upper bound of Theorem \ref{Thm:dimG}. It remains to prove that it is attained when $k$ is square-free.

\subsection{Lower bound for square-free \texorpdfstring{$k$}{k}}\label{Sec:SquareFree}

In this subsection, $k$ is square-free. In particular $16\nmid k$, so $n_d=\frac{k}{d}$, $\epsilon=1$ in Proposition \ref{Prop:depcarre}, and $\mathcal{D}(k)=\sum_{d\mid k}Card(\Lambda_{d})$. We only use roots of unity of odd order, and the idea is to read the linear independence of the Gauss sums off their behaviour near $1$. We first record two arithmetic facts.

\begin{lemma}\label{Lemma:Lemmeinverse}
Let $a\geq 1$. For $r$ coprime with $a$, let $I(r)$ be the integer in $\{-a+1,\ldots ,0\}$ congruent to $-r^{-1}$ modulo $a$. Then $I(r)$ only depends on $r$ modulo $a$, and $a^{\star}_{r}:=\frac{I(r)r+1}{a}$ is an inverse of $a$ modulo $r$.
\end{lemma}

\begin{proof}
By definition $a$ divides $I(r)r+1$, and $a\,a^{\star}_r=I(r)r+1\equiv 1 \pmod r$.
\end{proof}

\begin{lemma}\label{Lemma:GoesThrough}
Let $d$ be an odd divisor of $k$, $k_d=\frac{k}{d}$ and $I$ as above with $a=k_d$. For every $x\in\faktor{\mathbb{Z}}{k_d\mathbb{Z}}$, there are an odd integer $r$ coprime with $k_d$ and an integer $s$ coprime with $dr$ such that $sI(r)\equiv x \pmod{k_d}$.
\end{lemma}

\begin{proof}
As $k_d$ is square-free, $x=gcd(x,k_d).y$ with $y$ invertible modulo $k_d$. Take $r\equiv -y^{-1} \pmod{k_d}$, odd (add $k_d$ if needed), so that $I(r)\equiv y$. Since $k$ is square-free, $d$ and $k_d$ are coprime. Therefore $s\equiv gcd(x,k_d) \pmod{k_d}$ is coprime with $dr$.
\end{proof}

\subsubsection*{Decomposition of \texorpdfstring{$\mathcal{G}_k$}{G\_k}}
We will see that $\mathcal{G}_k$ decomposes as a direct sum of subspaces. In order to describe these subspaces, we will need to distinguish classes of integers associated to a divisor $d$ of $k$: 
\begin{definition}\label{def:dadmissible}
Let $d$ be an odd divisor of $k$ and $k_d=\frac{k}{d}$. We say that an integer $l$ is \emph{$d$-admissible} if it is coprime with $k_d$, or if $k$ is even and $l=2l'$ with $l'$ coprime with $\frac{k}{2d}$.
\end{definition}
As $k$ is square-free, the integers $0\leq l'<k$ whose gcd with $k$ has odd part $d$ are exactly the $dl$ with $l$ $d$-admissible (take $l'=dl$ and reduce modulo $k$), and no odd prime divisor of $k_d$ divides a $d$-admissible integer. 

We now define subspaces of $\mathcal{G}_k$.

\begin{definition}
Let $\mathcal{G}_{k,d}$ be defined as $$\mathcal{G}_{k,d} := Span_{\mathbb{Q} (A)} (\{\widehat{G}(k,2dl,\cdot)\}_{l\in L_d})$$
where $L_d$ is the set of $d$-admissible integers.
\end{definition}

It is clear that $\mathcal{G}_k=\sum_{d\ odd}\mathcal{G}_{k,d}$.

Let
$$N_d=Card(\Lambda_{k_d})+Card(\Lambda_{\frac{k}{2d}})\quad\textrm{($k$ even)},\qquad N_d=Card(\Lambda_{k_d})\quad\textrm{($k$ odd)},$$
and $\sum_{d\ odd}N_d=\sum_{d'\mid k}Card(\Lambda_{k/d'})=\mathcal{D}(k)$.

\begin{lemma}
The subspace $\mathcal{G}_{k,d}$ is spanned by $N_d$ functions $\widehat{G}(k,2dl_1,\cdot),\ldots ,\widehat{G}(k,2dl_{N_d},\cdot)$, where $l_1,\ldots ,l_{N_d}$ are $d$-admissible with pairwise distinct squares modulo $k_d$.
\end{lemma}
\begin{proof}
This is a direct consequence of Corollary \ref{Cor:UpperBounds} in the setting of $k$ square free.
\end{proof}

\subsubsection*{Linear independence}

Let $\mathbb{U}_d$ be the set of elements of $\mathbb{U}_{odd}$ whose order $c$ satisfies $gcd(k,c)=d$.

\begin{lemma}\label{Lemma:LemmaInterm}
Let $d$ be an odd divisor of $k$, $l\in\mathbb{Z}$, $r$ an odd integer coprime with $k_d$ and $s$ an integer coprime with $dr$. For $m\in\mathbb{N}$, let $R_m:=r+2msk_d$ and $\zeta_m:=e^{\frac{2is\pi}{dR_m}}$. Then $\zeta_m\in\mathbb{U}_d$, $\zeta_m\to 1$, $G(k,0,\zeta_m)\neq 0$ and
$$\dfrac{G(k,2dl ,\zeta_m )}{G(k, 0 , \zeta_m )} \underset{m\rightarrow\infty}{\longrightarrow} e^{\frac{-2i\pi s I(r) l^{2}}{k_d}}.$$
\end{lemma}

\begin{proof}
$R_m$ is odd and congruent to $r$ modulo $k_d$ and modulo every prime dividing $s$, so it is coprime with $k_d$ and $s$; thus $\zeta_m$ has odd order $dR_m$, $gcd(k,dR_m)=d$ and $\zeta_m\in\mathbb{U}_d$. The number $w:=\zeta_m^{k}$ is a primitive $R_m$-th root of unity, so $G(k,0,\zeta_m)=d\,G(1,0,w)\neq 0$ by Lemma \ref{Lemma:CalculGauss}(1). By Lemma \ref{Lemma:CalculGauss}(2) with $q=d$ and Lemma \ref{Lemma:Lemmeinverse}, as $I(R_m)=I(r)$,
$$\frac{G(k,2dl ,\zeta_m )}{G(k, 0 , \zeta_m )} = \zeta_m^{- d  l^{2} (k_d)^{\star}_{R_m }}=e^{ -\frac{2i\pi s  l^{2} (I(r)R_m+1) }{k_d R_m }} = e^{-\frac{2i\pi  s I(r )l^{2} }{k_d}}\, e^{ -\frac{2i\pi  sl^{2} }{R_m k_d}},$$
and the last factor tends to $1$.
\end{proof}

\begin{proposition}\label{Prop:indeplin}
Let $d$ be an odd divisor of $k$, and $l_1,\ldots ,l_{N}$ be $d$-admissible integers with pairwise distinct squares modulo $k_d$. If $R_1,\ldots ,R_N\in\mathbb{Q}(A)$ are such that $\sum_{j} R_j\, G(k,2dl_j ,\cdot ) = 0$ almost everywhere on $\mathbb{U}_d$, then all the $R_j$ are zero. In particular $\dim\mathcal{G}_{k,d}=N_d$, and a function of $\mathcal{G}_{k,d}$ which vanishes almost everywhere on $\mathbb{U}_d$ is zero.
\end{proposition}

\begin{proof}
Suppose that the $R_j$ are not all zero. Clearing denominators and dividing by a power of $A-1$, we may assume that they are polynomials which do not all vanish at $1$. Let $\zeta_m\in\mathbb{U}_d$ be as in Lemma \ref{Lemma:LemmaInterm}. Dividing the relation at $\zeta_m$ by $G(k,0,\zeta_m)$ and letting $m\to\infty$, we get
$$\sum\limits_{j=1}^{N} R_j (1)\, \omega^{x l_j^2}=0,\qquad \omega:=e^{-\frac{2i\pi}{k_d}},\quad x:=sI(r).$$
By Lemma \ref{Lemma:GoesThrough}, this holds for every $x\in\faktor{\mathbb{Z}}{k_d\mathbb{Z}}$. The numbers $\omega^{l_j^2}$ are pairwise distinct by assumption, and $N\leq k_d$, so the Vandermonde matrix $(\omega^{xl_j^2})_{0\leq x<N,\, j}$ is invertible. Hence all the $R_j(1)$ vanish, a contradiction.
\end{proof}

\begin{proposition}\label{Prop:dimdiviseur}
For $k$ square-free, $\mathcal{G}_{k} = \bigoplus_{d \mid k,\ d\ odd}\ \mathcal{G}_{k,d}$. Moreover, a function of $\mathcal{G}_k$ which vanishes almost everywhere on $\mathbb{U}_{odd}$ is zero.
\end{proposition}

\begin{proof}
Let $F=\sum_d F_d$ with $F_d\in\mathcal{G}_{k,d}$ vanish almost everywhere on $\mathbb{U}_{odd}$, and suppose that some $F_d$ is non-zero. Choose $d_0$ with $F_{d_0}\neq 0$ and $F_d=0$ for every odd multiple $d\neq d_0$ of $d_0$ dividing $k$. Let $\zeta\in\mathbb{U}_{d_0}$. If $d_0\nmid d$, there is a prime $p\mid d_0$ with $p\nmid d$, hence $p\mid k_d$ and $p$ does not divide any $d$-admissible integer $l$; thus $d_0\nmid 2dl$ and $G(k,2dl,\zeta)=0$ by Lemma \ref{Lemma:CalculGauss}(2) ($gcd(k,ord(\zeta))=d_0$). So every element of $\mathcal{G}_{k,d}$ vanishes at $\zeta$, and $0=F(\zeta)=F_{d_0}(\zeta)$ for almost all $\zeta\in\mathbb{U}_{d_0}$. By Proposition \ref{Prop:indeplin}, $F_{d_0}=0$, a contradiction. Both statements follow.
\end{proof}

\section{Proof of the main results}\label{Sec:Concl}

\begin{proof}[Proof of Theorem \ref{Thm:dimG}]
The upper bound is Corollary \ref{Cor:UpperBounds} together with Lemma \ref{Lemma:CardDiv}. If $k$ is square-free, Propositions \ref{Prop:indeplin} and \ref{Prop:dimdiviseur} give $\dim\mathcal{G}_k=\sum_{d\ odd}N_d=\mathcal{D}(k)$.
\end{proof}

The images of the evaluation map are not exactly elements of $\mathcal{G}_k$, but of $\mathcal{G}_k$ plus multiples of $\nu_0$ and $\nu_1$ (Theorem \ref{Thm:CalculRT}). We need to show that these extra terms are not in $\mathcal{G}_k$. We start with the following lemma, whose proof is a comparison of growth rates as $u\to 1$ along roots of unity of odd order.

\begin{lemma}\label{Lemma:FctConstante}
Let $R_0 ,\ldots ,R_N \in \mathbb{Q} (A)$ and $l_1 ,\ldots ,l_N \in \mathbb{Z}$ be such that, for almost all $\zeta \in\mathbb{U}_{odd}$,
$$R_0 (\zeta ) = \sum\limits^{N}_{j=1} R_j (\zeta ) G(k,2l_j ,\zeta  ).$$
Then $R_0=0$. In other words, the constant function $1$ is not in the $\mathbb{Q}(A)$-span of the $G(k,2l,\cdot)$ on $\mathbb{U}_{odd}$.
\end{lemma}

\begin{proof}
Multiplying the relation by a common denominator of the $R_j$ (nonzero at almost all $u$), we may assume
$R_0,\ldots,R_N\in\mathbb{Q}[A]$. Suppose $R_0\neq0$. Let $r_0>1$ be odd and coprime with $k$, and set
$r_m:=r_0+4mk$, $\zeta_m:=e^{\frac{2i\pi}{r_m}}\in\mathbb{U}_{odd}$. By Lemma~\ref{Lemma:CalculGauss}(1),
$G(k,0,\zeta_m)=\pm\epsilon_{r_m}\sqrt{r_m}$ where $\epsilon_{r_m}$ depends only on $r_m\equiv r_0\pmod 4$;
restricting to an infinite set of indices on which the sign is constant, $G(k,0,\zeta_m)=c\sqrt{r_m}$
with $c\in\mathbb{C}^\times$ constant. From now on $m$ ranges over this set, and limits are taken along it.

By Lemma~\ref{Lemma:CalculGauss}(2) and Lemma~\ref{Lemma:Lemmeinverse} (with $a=k$, $r_m\equiv r_0\pmod k$),
$G(k,2l_j,\zeta_m)=\zeta_m^{-l_j^2k^\star_{r_m}}G(k,0,\zeta_m)$ with $k^\star_{r_m}=\frac{I(r_0)r_m+1}{k}$, that is
$$G(k,2l_j,\zeta_m)=c\,e^{-\frac{2i\pi l_j^2I(r_0)}{k}}\,e^{-\frac{2i\pi l_j^2}{kr_m}}\sqrt{r_m}.$$

For $j\in\{1,\ldots,N\}$ the function $x\mapsto R_j(e^{2i\pi x})\,e^{-\frac{2i\pi l_j^2x}{k}}$ is entire,
and so is $x\mapsto R_0(e^{2i\pi x})$. Taking their Taylor series at $x=\frac1{r_m}$, we get
$\lambda_{j,t}\in\mathbb{C}$ ($t\geq0$, $0\leq j\leq N$) such that
$$R_j(\zeta_m)\,e^{-\frac{2i\pi l_j^2}{kr_m}}=\sum_{t\geq0}\lambda_{j,t}\,r_m^{-t}\ \ (1\leq j\leq N),
\qquad R_0(\zeta_m)=\sum_{t\geq0}\lambda_{0,t}\,r_m^{-t}.$$
Since $R_0\neq0$ has finitely many roots, $x\mapsto R_0(e^{2i\pi x})\not\equiv0$, so there is a least $d_0$ with
$\lambda_{0,d_0}\neq0$, and $R_0(\zeta_m)\sim\lambda_{0,d_0}\,r_m^{-d_0}$.

The hypothesis gives, for $m$ large,
$$R_0(\zeta_m)=\sum_{j=1}^NR_j(\zeta_m)G(k,2l_j,\zeta_m)=c\sqrt{r_m}\sum_{n\geq0}M_n\,r_m^{-n},
\qquad M_n:=\sum_{j=1}^Ne^{-\frac{2i\pi l_j^2I(r_0)}{k}}\lambda_{j,n}.$$
If $M_n=0$ for all $n$, then $R_0(\zeta_m)=0$ for infinitely many distinct $\zeta_m$, so $R_0=0$. Otherwise let $n$ be the least index with $M_n\neq0$. Then
$$\lambda_{0,d_0}\,r_m^{-d_0}\ \sim\ R_0(\zeta_m)\ \sim\ c\,M_n\,r_m^{\frac12-n}\qquad(m\to\infty),$$
with $\lambda_{0,d_0}\neq0$ and $c\,M_n\neq0$.
The left-hand
side has an integer power $-d_0\in\mathbb{Z}$ of $r_m$ while the right-hand side has a half-integer power
$\frac12-n\in\mathbb{Z}+\frac12$, which is a contradiction.
\end{proof}
\begin{lemma}\label{Lemma:FctConstante2}
Let $a,b\in\mathbb{Q}(A)$.
\begin{enumerate}
    \item If $a\nu_\epsilon\in\mathcal{G}_k$ for some $\epsilon\in\{0,1\}$, then $a=0$.
    \item If $k$ is square-free and $a\nu_0+b\nu_1\in\mathcal{G}_k$, then $a=b=0$.
\end{enumerate}
\end{lemma}

\begin{proof}
On $\mathbb{U}_{odd}$ we have $\delta=1$ and $\nu_\epsilon=1$. (1) Restricting a relation $a\nu_\epsilon=\sum_j R_j\widehat{G}(k,2l_j,\cdot)$ to $\mathbb{U}_{odd}$, Lemma \ref{Lemma:FctConstante} gives $a=0$. (2) The restriction of $a\nu_0+b\nu_1$ to $\mathbb{U}_{odd}$ is $a+b$, so $a+b=0$ by the same lemma. Hence $F:=a\nu_0+b\nu_1\in\mathcal{G}_k$ vanishes on $\mathbb{U}_{odd}$, and $F=0$ by Proposition \ref{Prop:dimdiviseur}. But for $r$ even, $F=a(\nu_0-\nu_1)=2a\,u^{km^2}$, and $u^{km^2}\neq 0$, so $a$ has infinitely many zeros: $a=0=b$.
\end{proof}

\begin{proof}[Proof of Theorem \ref{Thm:DimImages}]
By Theorem \ref{Thm:CalculRT} and Lemma \ref{Lemma:EvCasParticuliers} (with $q=2$ and $q=1$), for all $p\in\mathbb{Z}$,
\begin{align}
& ev((p,0)_T) = (-1)^p u^{-k}\left(\widehat{G}(k,2p,u) - \nu_{p \bmod 2}\right), \label{eq:evp0}\\
& ev((0,2)_T) = -\nu_{k \bmod 2}, \qquad ev((1,2)_T)=\nu_{(k+1) \bmod 2}, \label{eq:ev2}\\
& ev((0,1)_T)\neq 0\neq ev((1,1)_T). \label{eq:ev1}
\end{align}
Recall that $(p,q)_T$ lies in $S_{w}(M_k)$ with $w=px+qy$.

\emph{Case $k$ odd.} By Lemma \ref{lemma:Homology}, $h_k=\{\emptyset , y\}$, and by Theorem \ref{Thm:basis},
$$S_\emptyset (M_k)= Span \left( \{ (p,0)_T \ \vert \ 0\leq p \leq \left\lfloor \tfrac{k}{2} \right\rfloor \} \cup \{  (0,2)_T , (1,2)_T \} \right),\qquad S_y (M_k)= Span ( (0,1)_T ).$$
By \eqref{eq:ev1}, $\dim ev(S_y(M_k))=1$. By \eqref{eq:ev2}, $\nu_0,\nu_1\in ev(S_\emptyset (M_k))$, so by \eqref{eq:evp0} the functions $\widehat{G}(k,2p,\cdot)$, $0\leq p\leq \lfloor \frac{k}{2}\rfloor$, are also in $ev(S_\emptyset (M_k))$; they span $\mathcal{G}_k$ by Corollary \ref{Cor:UpperBounds}. Hence
$$ev(S_\emptyset (M_k)) = \mathcal{G}_k + \mathbb{Q}(A)\,\nu_0+\mathbb{Q}(A)\,\nu_1,$$
and $\dim ev(S_\emptyset (M_k))\leq \dim\mathcal{G}_k+2\leq \mathcal{D}(k)+2$ by Theorem \ref{Thm:dimG}. If $k$ is square-free, $\dim\mathcal{G}_k=\mathcal{D}(k)$ and Lemma \ref{Lemma:FctConstante2}(2) shows that $\nu_0,\nu_1$ are independent modulo $\mathcal{G}_k$, so we have equality.

\emph{Case $k$ even.} By Lemma \ref{lemma:Homology} and Remark \ref{Rk:BasisHorizontal}, $\mathcal{H}_k  = S_{\emptyset} \oplus S_{x} \oplus S_{y} \oplus S_{x+y}$, where $S_y$ and $S_{x+y}$ are generated by $(0,1)_T$ and $(1,1)_T$, and
\begin{align*}S_{\emptyset} &= Span \left( \{ (p,0)_T \ \vert \ p\ even,\ 0\leq p \leq  \tfrac{k}{2} \} \cup \{ (0,2)_T \} \right),\\
S_{x} &= Span \left( \{ (p,0)_T  \ \vert \ p\ odd,\ 0\leq p \leq \tfrac{k}{2} \} \cup \{ (1,2)_T \} \right).\end{align*}
By \eqref{eq:ev1}, $S_y$ and $S_{x+y}$ have images of dimension $1$. For $k$ even, \eqref{eq:ev2} reads $ev((0,2)_T)=-\nu_0$ and $ev((1,2)_T)=\nu_1$, so as in the odd case,
$$ev(S_\emptyset (M_k))=\mathcal{G}^{(0)}_k+\mathbb{Q}(A)\,\nu_0,\qquad ev(S_x (M_k))=\mathcal{G}^{(1)}_k+\mathbb{Q}(A)\,\nu_1,$$
where $\mathcal{G}^{(\epsilon)}_k$ is spanned by the $\widehat{G}(k,2p,\cdot)$ with $p\equiv\epsilon \pmod 2$. By Corollary \ref{Cor:UpperBounds}, the classes of collinear functions are made of functions with $p$ of the same parity, so $\dim\mathcal{G}^{(0)}_k+\dim\mathcal{G}^{(1)}_k\leq\mathcal{D}(k)$, and $\sum_{w \in h_k } \dim ev ( S_w (M_k) ) \leq \mathcal{D}(k)+ 4$.

Suppose now that $k\equiv 2 \pmod 4$ is square-free. By Lemma \ref{Lemma:FctConstante2}(1), $\nu_0\notin\mathcal{G}^{(0)}_k$ and $\nu_1\notin\mathcal{G}^{(1)}_k$, so $\dim ev(S_\emptyset)+\dim ev(S_x)=\dim\mathcal{G}^{(0)}_k+\dim\mathcal{G}^{(1)}_k+2$. Moreover, the basis of $\mathcal{G}_k$ given by Propositions \ref{Prop:indeplin} and \ref{Prop:dimdiviseur} consists of functions $\widehat{G}(k,2l,\cdot)$ with $l$ even or odd, and each $\widehat{G}(k,2l,\cdot)$ is collinear with one of them with $l$ of the same parity. Hence this basis splits into bases of $\mathcal{G}^{(0)}_k$ and $\mathcal{G}^{(1)}_k$, and $\dim\mathcal{G}^{(0)}_k+\dim\mathcal{G}^{(1)}_k=\dim\mathcal{G}_k=\mathcal{D}(k)$. We get the equality.
\end{proof}

\begin{proof}[Proof of Theorem \ref{Thm:Main}]
Let $\Sigma_k$ be the sum of the dimensions $\dim ev(S_w(M_k))$ over $w\in h_k$. By Theorem \ref{Thm:Kinnear}, the dimensions $\dim S_w(M_k)$ add up to $\dim\mathcal{H}_k$, and $\dim ev(S_w(M_k))\leq\dim S_w(M_k)$ for every $w$. So if $\Sigma_k<\dim\mathcal{H}_k$, then $ev$ is not injective on some $S_w(M_k)$, whereas if $\Sigma_k=\dim\mathcal{H}_k$, then it is injective on every $S_w(M_k)$.

\emph{Non-injectivity.} We compare the upper bounds of Theorem \ref{Thm:DimImages} with $\dim\mathcal{H}_k$ using Lemma \ref{Lemma:Inequalities}. If $k$ is odd with at least two prime divisors, then
$$\Sigma_k\leq\mathcal{D}(k)+3<\tfrac{k+1}{2}+3=\dim\mathcal{H}_k.$$
If $k\equiv 2 \pmod 4$ with at least three prime divisors, or if $4\mid k$ and $k\neq 4$, then
$$\Sigma_k\leq\mathcal{D}(k)+4<\tfrac{k}{2}+5=\dim\mathcal{H}_k.$$

\emph{Injectivity.} If $k\in\{1,2\}$ or $k=p$, $k=2p$ with $p$ an odd prime, then $k$ is square-free and we are in the equality case of Lemma \ref{Lemma:Inequalities}, so $\Sigma_k=\mathcal{D}(k)+3=\dim\mathcal{H}_k$ ($k$ odd) or $\Sigma_k=\mathcal{D}(k)+4=\dim\mathcal{H}_k$ ($k$ even), and $ev$ is injective on every $S_w(M_k)$.

Suppose $k=4$. Then $\mathcal{H}_4=S_\emptyset\oplus S_x\oplus S_y\oplus S_{x+y}$ with $S_\emptyset=Span((0,0)_T,(2,0)_T,(0,2)_T)$, $S_x=Span((1,0)_T,(1,2)_T)$, and $S_y$, $S_{x+y}$ of dimension $1$. On the latter two, $ev$ is injective by \eqref{eq:ev1}. On $S_x$, if $a\,ev((1,0)_T)+b\,ev((1,2)_T)=0$, then \eqref{eq:evp0} and \eqref{eq:ev2} give
$$a\widehat{G}(4,2,\cdot)=(a+bu^{4})\nu_1.$$
The left-hand side lies in $\mathcal{G}_4$, so $a+bu^4=0$ by Lemma \ref{Lemma:FctConstante2}(1), and $a\widehat{G}(4,2,\cdot)=0$. Since $G(4,2,u)\neq 0$ for $u$ of odd order (Lemma \ref{Lemma:CalculGauss}), $a=0$ and then $b=0$. On $S_\emptyset$, if $a\,ev((0,0)_T)+b\,ev((2,0)_T)+c\,ev((0,2)_T)=0$, then similarly
$$a\widehat{G}(4,0,\cdot)+b\widehat{G}(4,4,\cdot)=(a+b+cu^4)\nu_0,$$
so $a+b+cu^4=0$ and $a\widehat{G}(4,0,\cdot)+b\widehat{G}(4,4,\cdot)=0$. If $u$ has order $8r'$ with $r'$ odd, then $G(4,0,u)=0\neq G(4,4,u)$ by Lemma \ref{Lemma:CalculGauss}(1),(3), so $b$ has infinitely many zeros and $b=0$. Then $G(4,0,u)\neq 0$ for $u$ of odd order gives $a=0$, and finally $c=0$.
\end{proof}

\begin{remark}\label{Rk:OddOnly}
All the independence arguments of Section \ref{Sec:SquareFree} and Lemma \ref{Lemma:FctConstante} only use roots of unity $u$ of odd order, that is, values $ev ( (l,0)_T) (\zeta )$ where $\zeta$ is a primitive $2r$-th root of unity with $r$ odd. These independence statements therefore only involve the $SO_3$-invariants. Roots of unity of even order are only needed to separate $\nu_0$ from $\nu_1$ (Lemma \ref{Lemma:FctConstante2}(2)) and to prove the injectivity of $ev$ for $k=4$.
\end{remark}
\section{Formulas for a general monodromy}\label{sec:GeneralFormulas}

We sketch here how to compute the $RT_{\zeta}$ invariants of pairs $(M_B , L)$ where $B$ is a general monodromy and $L$ is a multicurve on $\mathbb{T}^2$.\\
For simplicity, we only consider $\zeta$ of order $2r$ with $r$ odd, and we set $u:=-\zeta$, a primitive root of unity of order $r$. We can get similar formulas when $r$ is even, but the proof is more technical and requires more case by case analysis.

Since the action of the multicurve is already computed in Lemma \ref{Lemma:Action(p,q)}, the only thing that remains to compute is the action associated to the monodromy :

\begin{proposition}\label{Prop:GeneralFormulas}
Let $\zeta\in\mathbb{U}_0$ be of order $2r$ with $r$ odd and $u=-\zeta$.\smallbreak
\noindent For $\begin{pmatrix}
a & b \\
c & d \\
\end{pmatrix} \in SL_2 (\mathbb{Z} ) $, set $g:=gcd(r,c)$ (so $g=r$ if $c=0$) and $r':=\dfrac{r}{g}$.\smallbreak
Then, up to multiplication by a root of unity (which depends on the matrix and on $\zeta$, but not on $j$):
$$
\rho \left( \begin{pmatrix}
a & b \\
c & d \\
\end{pmatrix} \right) e_j = \dfrac{1}{\sqrt{r'}} \sum\limits_{t\in\faktor{\mathbb{Z}}{r'\mathbb{Z}}} (-1)^{(d+1)j+t}\, u^{a(\frac{c}{g})^{\star}_{r'} gt^2 + bdj^2 + 2bgjt} e_{dj+gt}
$$
where $( \frac{c}{g} )^{\star}_{r'}$ is an inverse of $\frac{c}{g}$ modulo $r'$ (it exists since $\gcd(\frac{c}{g},r')=1$; it is $0$ if $r'=1$).
\end{proposition}

\begin{remark}
Since the TQFT is defined up to a root of unity called anomaly, the fact that we are performing these computations up to a multiplication by a root of unity still yields a formula for the WRT-invariant.
\end{remark}

\begin{proof}
First of all, recall that $SL_2 (\mathbb{Z} )$ is generated by $S:=\begin{pmatrix}
0 & -1 \\
1 & 0 \\
\end{pmatrix} $ and $T:=\begin{pmatrix}
1 & 1 \\
0 & 1 \\
\end{pmatrix} $.

It is stated in \cite[\S 2]{GilRep} (with $A_r=\zeta$) that, for $1\leq i\leq \frac{r-1}{2}$,
$$\rho \left( T \right) e_i = u^{i^2} e_i,\qquad \rho \left( S \right) e_i = \dfrac{1}{\sqrt{r}} \sum\limits_{t\in\faktor{\mathbb{Z}}{r\mathbb{Z}}} (-1)^{i+t} u^{2ti} e_{t}= \dfrac{-1}{\sqrt{r}} \sum\limits_{t\in\faktor{\mathbb{Z}}{r\mathbb{Z}}} (-1)^{i+t} u^{-2ti} e_{t}
$$
Then, up to multiplication by $-1$, the formula of the proposition for $\rho : SL_2 (\mathbb{Z} ) \rightarrow End (RT_\zeta (\mathbb{T}^2 ))$ is correct on the generators of $SL_2 (\mathbb{Z} )$. 

For $M=\begin{pmatrix}
a & b \\
c & d \\
\end{pmatrix}$, denote by $\rho (M)$ the operator defined by the right-hand side of the proposition.

We now need to show that, if the proposition holds for $M$, then it holds for $TM$ and for $SM$. 
In other words, we will show that $\rho(T)\rho (M)=\rho (TM)$ and that $\rho(S)\rho (M)=\lambda\,\rho (SM)$ for a root of unity $\lambda$ (independent of $j$).

Suppose the result true for $M= \begin{pmatrix}
a & b \\
c & d \\
\end{pmatrix} $.

Studying the composition on the left by $T$ gives (note that $TM$ has the same $c$ and $d$, hence the same $g$, $r'$ and signs, as $M$):
\begin{align*}
\rho \left( T \right) \rho (M) e_j &\equiv \dfrac{1}{\sqrt{r'}} \sum\limits_{t\in\faktor{\mathbb{Z}}{r'\mathbb{Z}}} (-1)^{(d+1)j+t}\, u^{a(\frac{c}{g})^{\star}_{r'} gt^2 + bdj^2 + 2bgjt} \rho \left( T \right) e_{dj+gt}\\
&= \dfrac{1}{\sqrt{r'}} \sum\limits_{t\in\faktor{\mathbb{Z}}{r'\mathbb{Z}}} (-1)^{(d+1)j+t}\, u^{a(\frac{c}{g})^{\star}_{r'} gt^2 + bdj^2 + 2bgjt +(dj+gt)^2}  e_{dj+gt}\\
&= \dfrac{1}{\sqrt{r'}} \sum\limits_{t\in\faktor{\mathbb{Z}}{r'\mathbb{Z}}} (-1)^{(d+1)j+t}\, u^{(a+c)(\frac{c}{g})^{\star}_{r'} gt^2 + (b+d)dj^2 + 2(b+d)gjt}  e_{dj+gt}
\end{align*}
Indeed, $(\frac{c}{g})^\star_{r'}\frac{c}{g}=1+r'm$ for some integer $m$, so that $c(\frac{c}{g})^\star_{r'}gt^2=g^2t^2+g\,rm\,t^2\equiv g^2t^2 \pmod r$.
The last line is what is expected of $\rho (TM) e_j$, where $TM=\begin{pmatrix}
a+c & b+d \\
c & d \\
\end{pmatrix}$.

And the left composition by $S$ gives, with $SM=\begin{pmatrix}
-c & -d \\
a & b \\
\end{pmatrix}$:
$$
\rho \left( S \right) \rho (M) e_j
=  \dfrac{1}{\sqrt{r'}}\dfrac{1}{\sqrt{r}} \sum\limits_{t\in\faktor{\mathbb{Z}}{r'\mathbb{Z}}} \sum\limits_{s\in\faktor{\mathbb{Z}}{r\mathbb{Z}}}  (-1)^{j+s}\, u^{a(\frac{c}{g})^{\star}_{r'} gt^2 + bdj^2 + 2bgjt + 2 (dj+gt)s} e_{s}
$$
Here we used the formula for $\rho(S)$ with the arbitrary index $x=dj+gt$, and the signs have combined into $(-1)^{(d+1)j+t}(-1)^{dj+gt+s}=(-1)^{j+s}$, because $g$ is odd. In particular, they do not depend on $t$.

Since $gcd(a,c)=1$, one can set $g_2:= gcd(a,r) = gcd(a,r')$ (it is odd because $r$ is odd), $\tilde{r} := \dfrac{r'}{g_2}$ and $\overline{r}:=\dfrac{r}{g_2}=g\tilde{r}$.

Because $w:=u^g$ is a primitive root of unity of order $r'$ and because $g_2$ is odd, Lemma \ref{Lemma:CalculGauss} (with $q=g_2$) gives that, for a fixed $s$, the sum over $t$ is
\begin{align*}
\sum\limits_{t\in\faktor{\mathbb{Z}}{r'\mathbb{Z}}}  w^{a(\frac{c}{g})^{\star}_{r'} t^2 + 2(bj+s)t} &= G(a(\frac{c}{g})^{\star}_{r'}, 2(bj+s) ,w)\\
&= \left\lbrace\begin{array}{ll}
g_2\, \epsilon' \sqrt{\tilde{r}}\; u^{-c\,g_2 \tau^2 (\frac{a}{g_2} )^\star_{\tilde{r}}} & \textrm{if } bj+s=g_2\tau,\ \tau\in\mathbb{Z},\\
0 & \textrm{if } g_2\nmid bj+s,
\end{array}\right.
\end{align*}
where $\epsilon'=\pm\epsilon_{\tilde{r}}$ is a root of unity which depends on $M$ and $\zeta$ but neither on $j$ nor on $\tau$. 

Thus, since $\{ bj+s \ \vert \ s\in\faktor{\mathbb{Z}}{r\mathbb{Z}} \} \simeq \faktor{\mathbb{Z}}{r\mathbb{Z}}$, keeping only the values of $s$ such that $bj+s =g_2 \tau$, that is $s=g_2\tau-bj$ with $\tau\in\faktor{\mathbb{Z}}{\overline{r}\mathbb{Z}}$
, we have that
\begin{align*}
\rho \left( S \right) \rho (M) e_j
&= \dfrac{\epsilon'}{\sqrt{\overline{r}}}  \sum\limits_{\tau\in\faktor{\mathbb{Z}}{\overline{r}\mathbb{Z}}}  (-1)^{(b+1)j+\tau}\,   u^{-bdj^2 +  2dg_2 \tau j- g_2\tau^2 c(\frac{a}{g_2} )^\star_{\tilde{r}} }  e_{g_2 \tau-bj}.
\end{align*}
With the change of variable $\tau\mapsto-\tau$ (which does not change $(-1)^\tau$) and since $e_{-n}=-e_n$, this gives
\begin{align*}
\rho \left( S \right) \rho (M) e_j &= -\dfrac{\epsilon'}{\sqrt{\overline{r}}}  \sum\limits_{\tau\in\faktor{\mathbb{Z}}{\overline{r}\mathbb{Z}}}  (-1)^{(b+1)j+\tau}\,   u^{-bdj^2 -  2dg_2 \tau j- g_2\tau^2 c(\frac{a}{g_2} )^\star_{\tilde{r}}} e_{g_2 \tau+ bj}
\end{align*}
Finally, the lower-left entry of $SM$ is $a$, so the corresponding $g$ is $g_2=\gcd(r,a)$ and the corresponding $r'$ is $\overline{r}$; since its lower-right entry is $b$, the corresponding sign is $(-1)^{(b+1)j+\tau}$, and
$$\rho (SM)e_j=\dfrac{1}{\sqrt{\overline{r}}}\sum\limits_{\tau\in\faktor{\mathbb{Z}}{\overline{r}\mathbb{Z}}}(-1)^{(b+1)j+\tau}\,u^{-c(\frac{a}{g_2})^\star_{\overline{r}}g_2\tau^2-bdj^2-2dg_2j\tau}e_{bj+g_2\tau}.$$
Since $\tilde{r}$ divides $(\frac{a}{g_2} )^\star_{\tilde{r}} - (\frac{a}{g_2} )^\star_{\overline{r}}$ and $c\,g_2\tilde{r}=\frac{c}{g}\,r$ is a multiple of $r$, the two exponents of $u$ agree. Hence $\rho(S)\rho (M)=-\epsilon'\,\rho (SM)$, with $-\epsilon'$ a root of unity, which is what is expected of
$$\rho \left( S M \right) = \rho \left( \begin{pmatrix}
-c & -d \\
a & b \\
\end{pmatrix} \right).$$
\end{proof}

\section{The mapping torus of the monodromy \texorpdfstring{$\begin{pmatrix}
0 & -1 \\
1 & 0 \\
\end{pmatrix} $}{S}}\label{SubSec:S_Matrix}

Let $S :=\begin{pmatrix}
0 & -1 \\
1 & 0 \\
\end{pmatrix} $. As for $M_k$, let $h_S$ be the horizontal part of $H_1 (M_S ,\faktor{\mathbb{Z}}{2\mathbb{Z}} )$ and $\mathcal{H}_S\subset S_A (M_S,\mathbb{Q} (A))$ the corresponding subspace. In this section, we compute a basis of $\mathcal{H}_S$ and the images of its elements under $ev_{M_S}$.

\subsection*{A basis of \texorpdfstring{$\mathcal{H}_S$}{H\_S}}

\begin{lemma}\label{Lemma:Formulasbis}
Let $p,q\in\mathbb{Z}$. In $\mathcal{H}_S$,
\begin{equation}\tag{c}\label{Dependance3}
A^{p} (p,q+1)_T + A^{-p} (p,q-1)_T = A^{-q} (p-1,q)_T + A^{q} (p+1,q)_T,
\end{equation}
\begin{equation}\tag{d}\label{Dependance4}
A^{p-q} (p+1,q+1)_T + A^{q-p} (p-1,q-1)_T = A^{-p-q} (p-1,q+1)_T + A^{p+q} (p+1,q-1)_T.
\end{equation}
\end{lemma}

\begin{proof}
As for Lemma \ref{Lemma:Formulas}, isotoping a skein along the $\mathbb{S}^1$ factor gives $(p,q)_T = S((p,q)_T) = (-q, p)_T $ in $\mathcal{H}_S$. Hence $(p,q)_T \star (0,1)_T = (-1,0)_T \star (p,q)_T$ and $(p,q)_T \star (1,1)_T = (-1,1)_T \star (p,q)_T$, and the product-to-sum formula gives (\ref{Dependance3}) and (\ref{Dependance4}) respectively.
\end{proof}

We deduce a set of generators for $\mathcal{H}_S$:

\begin{proposition}\label{Prop:BasisHS}
$\{ (0,0)_T ,\ (1,0)_T,\ (1,1)_T,\ (1,2)_T \}$ generates $\mathcal{H}_S$.
\end{proposition}

\begin{proof}
Relation (\ref{Dependance4}) expresses $(p,q)_T$ with $\vert p\vert ,\vert q\vert >1$ in terms of elements with smaller $\vert p\vert +\vert q\vert$, and relation (\ref{Dependance3}) does the same for $\vert p\vert\leq 1$ and $\vert q\vert >2$ (applying it twice if necessary). Since $(p,q)_T=(-p,-q)_T$, $\mathcal{H}_S$ is spanned by the $(p,q)_T$ with $\vert p\vert ,\vert q\vert \leq 2$. Applying (\ref{Dependance3}) and (\ref{Dependance4}) to the finitely many remaining cases, we see that all of them are combinations of the four announced elements. For instance, (\ref{Dependance3}) with $(p,q)=(0,1)$ gives $(0,2)_T=(A+A^{-1})(1,1)_T-(0,0)_T$.
\end{proof}

\begin{remark}\label{Rk:GradedS}
The monodromy $S$ exchanges the two generators of $H_1(\mathbb{T}^2,\faktor{\mathbb{Z}}{2\mathbb{Z}})$, so their images in $H_1(M_S,\faktor{\mathbb{Z}}{2\mathbb{Z}})$ coincide; call $x$ this class. Then $h_S=\{\emptyset , x\}$, the class of $(p,q)_T$ is $(p+q)x$, and the graded subspaces are $S_\emptyset (M_S) = Span \{ (0,0)_T, \ (1,1)_T \} $ and $S_x (M_S) = Span \{ (1,0)_T ,\ (1,2)_T \}$.
\end{remark}

\subsection*{The evaluation map on \texorpdfstring{$\mathcal{H}_S$}{H\_S}}
Let $\zeta\in \mathbb{U}_0$ be of order $2r$ with $r$ odd, $u=-\zeta$ (of order $r$), $n=\frac{r-1}{2}$, $2^{\star}_r=\frac{r+1}{2}$ (so $2.2^{\star}_r\equiv1\pmod r$), 
. By the first item of Lemma \ref{Lemma:CalculGauss}, $G(2,0,u)\in\{\pm\sqrt r\}$ if $r\equiv1\pmod4$ and $G(2,0,u) \in\{\pm i\sqrt r\}$ if $r\equiv3\pmod4$; in particular $G(2,0,u) \ne0$ and $G(-2,0,u)=\overline{G(2,0,u)}=(-1)^n G(2,0,u)$.

\begin{proposition}
For every $(p,q)\in\mathbb Z^2$ and almost all $\zeta$
\[
ev((p,q)_T)(\zeta)=(-1)^{p+q+pq}\,\frac{1}{\sqrt r} G(2,0,u)\,\big(u^{-2^{\star}_r(p^2+q^2)}-(-1)^n u^{2^{\star}_r(p^2+q^2)}\big) .
\]
\end{proposition}

\begin{proof}
By the trace formula, $ev((p,q)_T)(\zeta)=\mathrm{Tr}\big(\rho(S)\circ Z((p,q)_T)\big)$ in the basis $(e_j)_{1\le j\le n}$ (with $e_{-m}=-e_m$, $e_{m+r}=-e_m$, $e_0=0$).

By [Gil99a, \S2], $\rho(S)e_i=\frac1{\sqrt r}\sum_{t=1}^{n}(-1)^{i+t}\big(u^{2it}-u^{-2it}\big) e_t$.

By Lemma \ref{Lemma:Action(p,q)}, $Z((p,q)_T)e_j=(-1)^p\big(\zeta^{2pj+pq}e_{j+q}+\zeta^{-2pj+pq}e_{j-q}\big)$. Hence, for $r>q$
$$ev((p,q)_T) = \frac{(-1)^{p+q}}{\sqrt r} \sum\limits_{j=1}^{n} F_{p,q}(j) ,$$
with
\[
F_{p,q}(j)=\zeta^{2pj+pq}\big(u^{2j(j+q)}-u^{-2j(j+q)}\big)+\zeta^{-2pj+pq}\big(u^{2j(j-q)}-u^{-2j(j-q)}\big),
\]
which is even, $r$-periodic and vanishes at $0$ ($\zeta^{2r}=1$). Since $-n,\dots,n$ is a complete system of residues mod~$r$, $\sum_{j=1}^nF_{p,q}(j)=\frac12\sum\limits_{j\in \faktor{\mathbb{Z}}{r\mathbb{Z}}}F_{p,q}(j)$.

Use $\zeta^2=u^2$, $\zeta^{pq}=(-1)^{pq}u^{pq}$ and $j\mapsto-j$ in two of the four terms:
\[
\sum_{j\in\faktor{\mathbb{Z}}{r\mathbb{Z}}}F_{p,q}(j)=2(-1)^{pq}u^{pq}\Big(G\big(2,2(p+q),u\big)-G\big(-2,2(p-q),u\big)\Big).
\]
Completing the square mod $r$ (with $2. 2^{\star}_r \equiv 1$ and $b\in \{ p+q , p-q \}$) gives $\pm 2j^2+2bj\equiv \pm 2(j\pm b.2^{\star}_r)^2\mp b^2.2^{\star}_r$
, so
\[
G(2,2b,u)=u^{-2^{\star}_rb^2}G(2,0,u),\qquad G(-2,2b,u)=u^{2^{\star}_rb^2}\,\overline{G(2,0,u)}=(-1)^n\,u^{2^{\star}_rb^2}G(2,0,u) .
\]
Since $u^{pq}\big(u^{-2^{\star}_r(p+q)^2}-(-1)^n u^{2^{\star}_r(p-q)^2}\big)=u^{-2^{\star}_r(p^2+q^2)}-(-1)^n u^{2^{\star}_r(p^2+q^2)}$, collecting the factors $\frac{(-1)^{p+q}}{\sqrt r}\cdot\frac12\cdot2(-1)^{pq}$ gives the formula.
\end{proof}


\begin{corollary}\label{Cor:evS}
With $c:=\frac{G(2,0,u)}{\sqrt r}$,
\begin{align*}
ev((0,0)_T)(\zeta)&=c\,(1-(-1)^n), & ev((1,0)_T)(\zeta)&=-c\,u^{-2^\star_r}(1-(-1)^nu),\\
ev((1,1)_T)(\zeta)&=-c\,(u^{-1}-(-1)^nu), & ev((1,2)_T)(\zeta)&=-c\,u^{2-2^\star_r.9}(1-(-1)^nu^{5}).
\end{align*}
\end{corollary}

\begin{proposition}\label{Prop:NonColinearS}
Over $\mathbb{Q}(A)$, the functions $ev_{M_S} ((1,0)_T)$ and $ev_{M_S} ((1,2)_T)$ are not collinear, and neither are $ev_{M_S} ((0,0)_T)$ and $ev_{M_S} ((1,1)_T)$. Already the values at the $\zeta$ of order $2r$ with $r$ odd distinguish them.
\end{proposition}

\begin{proof}
Suppose $a\, ev((1,0)_T)+b\, ev((1,2)_T)=0$ for almost all $\zeta\in U_0$ of odd order, with $a,b\in\mathbb Q(A)$ not both zero. By Corollary \ref{Cor:evS}, $ev((1,0)_T)\ne0$, so $b\ne0$. By Corollary \ref{Cor:evS} and $u^{-2^\star_r 8+2}=u^{-2}$, for almost all such $\zeta$,
\[
\frac{a}{b}=\frac{u^{-2}- (-1)^n u^{3}}{1- (-1)^n u}.
\]
Then, for $r\equiv1\pmod4$, $\frac{a}{b} (u) = \frac{u^{-2}-u^3}{1-u}$ and for $r\equiv3\pmod4$, $\frac{a}{b} (u) = \frac{u^{-2}+u^3}{1+u}$, both for infinitely many $\zeta $ so $\frac{u^{-2}-u^3}{1-u} =\frac{u^{-2}+u^3}{1+u}$. But $(u^{-2}-u^3)(1+u)-(u^{-2}+u^3)(1-u)=2u^{-1}-2u^3\ne0$, a contradiction.

\smallskip
Suppose $a ev((0,0)_T)+b\, ev((1,1)_T)=0$ for almost all $\zeta\in U_0$ of odd order. For $r\equiv1\pmod4$, $ev((0,0)_T)=0$ and $ev((1,1)_T)\neq0$ by Corollary \ref{Cor:evS}, so $b$ has infinitely many zeros and $b=0$. The relation becomes $a\,ev((0,0)_T)=0$, and for $r\equiv3\pmod4$, $ev((0,0)_T)=2c\ne0$ by~(8.8); so $a$ has infinitely many zeros and $a=0$.

\end{proof}

\begin{corollary}\label{Cor:InjS}
The evaluation map $ev_{M_S}$ is injective on $S_\emptyset (M_S)$ and on $S_x(M_S)$, hence on each graded subspace of $\mathcal{H}_S$. Moreover, the set of Proposition \ref{Prop:BasisHS} is a basis for $\mathcal{H}_S$.
\end{corollary}

 
\bibliographystyle{hamsalpha}
\bibliography{biblio}

@misc{BelDet,
      title={An effective proof of finiteness for Kauffman bracket skein modules}, 
      author={Giulio Belletti and Renaud Detcherry},
      year={2025},
      eprint={2507.02589},
      archivePrefix={arXiv},
      primaryClass={math.GT},
      note={\url{https://arxiv.org/abs/2507.02589}}, 
}

@misc{GJV,
      title={Skeins on tori}, 
      author={Sam Gunningham and David Jordan and Monica Vazirani},
      year={2024},
      eprint={2409.05613},
      archivePrefix={arXiv},
      primaryClass={math.QA},
      note={\url{https://arxiv.org/abs/2409.05613}}, 
}

@book{BKE,
  title={Gauss and {J}acobi sums},
  author={Berndt, Bruce C. and Evans, Ronald J. and Williams, Kenneth S.},
  series={Canadian Mathematical Society Series of Monographs and Advanced Texts},
  publisher={Wiley-Interscience, John Wiley \& Sons, Inc., New York},
  year={1998},
}

@article {BHMVInvariant,
    AUTHOR = {Blanchet, C. and Habegger, N. and Masbaum, G. and Vogel, P.},
     TITLE = {Three-manifold invariants derived from the {K}auffman bracket},
   JOURNAL = {Topology},
  FJOURNAL = {Topology. An International Journal of Mathematics},
    VOLUME = {31},
      YEAR = {1992},
    NUMBER = {4},
     PAGES = {685--699},
      ISSN = {0040-9383},
   MRCLASS = {57M25 (57N10)},
  MRNUMBER = {1191373},
MRREVIEWER = {Christine\ Lescop},
       DOI = {10.1016/0040-9383(92)90002-Y},
       note={\url{https://doi.org/10.1016/0040-9383(92)90002-Y}},
}

@article {DetWol,
    AUTHOR = {Detcherry, Renaud and Wolff, Maxime},
     TITLE = {A basis for the {K}auffman skein module of the product of a
              surface and a circle},
   JOURNAL = {Algebr. Geom. Topol.},
  FJOURNAL = {Algebraic \& Geometric Topology},
    VOLUME = {21},
      YEAR = {2021},
    NUMBER = {6},
     PAGES = {2959--2993},
      ISSN = {1472-2747,1472-2739},
   MRCLASS = {57K31},
  MRNUMBER = {4344875},
MRREVIEWER = {Qi\ Chen},
       DOI = {10.2140/agt.2021.21.2959},
       note={\url{https://doi.org/10.2140/agt.2021.21.2959}},
}

@article {FroGel,
    AUTHOR = {Frohman, Charles and Gelca, R\u{a}zvan},
     TITLE = {Skein modules and the noncommutative torus},
   JOURNAL = {Trans. Amer. Math. Soc.},
  FJOURNAL = {Transactions of the American Mathematical Society},
    VOLUME = {352},
      YEAR = {2000},
    NUMBER = {10},
     PAGES = {4877--4888},
      ISSN = {0002-9947,1088-6850},
   MRCLASS = {57M25 (46L87 57M27 58B34)},
  MRNUMBER = {1675190},
MRREVIEWER = {Justin\ D.\ Roberts},
       DOI = {10.1090/S0002-9947-00-02512-5},
       note={\url{https://doi.org/10.1090/S0002-9947-00-02512-5}},
}

@article {GilRep,
    AUTHOR = {Gilmer, Patrick M.},
     TITLE = {On the {W}itten-{R}eshetikhin-{T}uraev representations of
              mapping class groups},
   JOURNAL = {Proc. Amer. Math. Soc.},
  FJOURNAL = {Proceedings of the American Mathematical Society},
    VOLUME = {127},
      YEAR = {1999},
    NUMBER = {8},
     PAGES = {2483--2488},
      ISSN = {0002-9939,1088-6826},
   MRCLASS = {57N10 (57M99)},
  MRNUMBER = {1487369},
MRREVIEWER = {Gretchen\ Wright},
       DOI = {10.1090/S0002-9939-99-04796-6},
       note={\url{https://doi.org/10.1090/S0002-9939-99-04796-6}},
}

@article {Gil,
    AUTHOR = {Gilmer, Patrick M.},
     TITLE = {On the {K}auffman bracket skein module of the 3-torus},
   JOURNAL = {Indiana Univ. Math. J.},
  FJOURNAL = {Indiana University Mathematics Journal},
    VOLUME = {67},
      YEAR = {2018},
    NUMBER = {3},
     PAGES = {993--998},
      ISSN = {0022-2518,1943-5258},
   MRCLASS = {57R56 (57M27)},
  MRNUMBER = {3820238},
MRREVIEWER = {Xin\ Liu},
       DOI = {10.1512/iumj.2018.67.7327},
       note={\url{https://doi.org/10.1512/iumj.2018.67.7327}},
}

@article {GilMas,
    AUTHOR = {Gilmer, Patrick M. and Masbaum, Gregor},
     TITLE = {On the skein module of the product of a surface and a circle},
   JOURNAL = {Proc. Amer. Math. Soc.},
  FJOURNAL = {Proceedings of the American Mathematical Society},
    VOLUME = {147},
      YEAR = {2019},
    NUMBER = {9},
     PAGES = {4091--4106},
      ISSN = {0002-9939,1088-6826},
   MRCLASS = {57M27 (57N10 57R56)},
  MRNUMBER = {3993800},
MRREVIEWER = {Xin\ Liu},
       DOI = {10.1090/proc/14553},
       note={\url{https://doi.org/10.1090/proc/14553}},
}

@article {GJS,
    AUTHOR = {Gunningham, Sam and Jordan, David and Safronov, Pavel},
     TITLE = {The finiteness conjecture for skein modules},
   JOURNAL = {Invent. Math.},
  FJOURNAL = {Inventiones Mathematicae},
    VOLUME = {232},
      YEAR = {2023},
    NUMBER = {1},
     PAGES = {301--363},
      ISSN = {0020-9910,1432-1297},
   MRCLASS = {57K31},
  MRNUMBER = {4557403},
MRREVIEWER = {Mohamed\ Elhamdadi},
       DOI = {10.1007/s00222-022-01167-0},
       note={\url{https://doi.org/10.1007/s00222-022-01167-0}},
}

@article{Kin,
   title={Skein module dimensions of mapping tori of the 2-torus},
   ISSN={1664-073X},
   note={\url{http://dx.doi.org/10.4171/QT/231}},
   DOI={10.4171/qt/231},
   journal={Quantum Topology},
   publisher={European Mathematical Society - EMS - Publishing House GmbH},
   author={Kinnear, Patrick},
   year={2025},
   month=Jan}

@article {Prz91,
    AUTHOR = {Przytycki, J\'{o}zef H.},
     TITLE = {Skein modules of {$3$}-manifolds},
   JOURNAL = {Bull. Polish Acad. Sci. Math.},
  FJOURNAL = {Bulletin of the Polish Academy of Sciences. Mathematics},
    VOLUME = {39},
      YEAR = {1991},
    NUMBER = {1-2},
     PAGES = {91--100},
      ISSN = {0239-7269},
   MRCLASS = {57M25 (57N10)},
  MRNUMBER = {1194712},
}

@article {Tur,
    AUTHOR = {Turaev, V. G.},
     TITLE = {The {C}onway and {K}auffman modules of a solid torus},
   JOURNAL = {Zap. Nauchn. Sem. Leningrad. Otdel. Mat. Inst. Steklov.
              (LOMI)},
  FJOURNAL = {Zapiski Nauchnykh Seminarov Leningradskogo Otdeleniya
              Matematicheskogo Instituta imeni V. A. Steklova Akademii Nauk
              SSSR (LOMI)},
    VOLUME = {167},
      YEAR = {1988},
     PAGES = {79--89, 190},
      ISSN = {0373-2703},
   MRCLASS = {57M25},
  MRNUMBER = {964255},
MRREVIEWER = {J\'{o}zef\ H.\ Przytycki},
       DOI = {10.1007/BF01099241},
       note={\url{https://doi.org/10.1007/BF01099241}},
}
\end{document}